\documentclass{amsart}

\usepackage[left=1.2in,right=1in,top=1in,bottom=.8in]{geometry}

\usepackage[utf8]{inputenc}
\usepackage[english]{babel}

\usepackage{graphicx}
\graphicspath{ {./images/} }
\usepackage{amsmath, lipsum}
\usepackage{amssymb}
\usepackage{amsthm}
\usepackage{framed}
\usepackage{mathtools}
\usepackage[colorlinks]{hyperref}
\usepackage{enumerate}
\usepackage{enumitem}
\usepackage{booktabs}
\usepackage{array}
\usepackage{faktor}
\usepackage{cancel}
\usepackage{makecell}
\usepackage{pbox}
\usepackage{bm}
\usepackage[toc,page]{appendix}
\usepackage{etoolbox}
\usepackage{stmaryrd}
\usepackage{xstring}
\usepackage{extarrows}
\usepackage{blindtext}
\usepackage{comment}
\usepackage{centernot}
\usepackage{floatrow}
\usepackage{leftindex}
\allowdisplaybreaks[1]
\usepackage[dvipsnames]{xcolor}

\usepackage{tikz}
\usepackage{tikz-cd}
\usepackage{quiver}
\usetikzlibrary{arrows.meta}
\usepackage[all,cmtip]{xy}
\usetikzlibrary{patterns}
\usetikzlibrary{cd}
\usetikzlibrary{calc}

\newcommand{\da}{\partial}

\newcommand{\rr}{\rightrightarrows}

\newcommand{\tg}{{\rm t}}
\newcommand{\m}        {\rm m}
\newcommand{\iv}        {\rm i}
\newcommand{\s}{{\rm s}}
\newcommand{\inv}{^{-1}}
\newcommand{\an}[1]{\arrowvert_{#1}}

\newcommand{\lrtimes}[2]{\leftindex[I]_{#1}\hspace{-0.2mm}\times_{#2}}

\renewcommand{\tilde}[1]{\widetilde{#1}}

\newcommand\arrowsquare{} 
\def\arrowsquare[#1,#2,#3,#4,#5,#6]{
	\pgfmathsetmacro{\xleft}{#1}
	\pgfmathsetmacro{\xmid}{#1 + 0.5 * #3}
	\pgfmathsetmacro{\xright}{#1 + #3}
	\pgfmathsetmacro{\ytop}{#2}
	\pgfmathsetmacro{\ymid}{#2 + 0.5 * #4}
	\pgfmathsetmacro{\ybot}{#2 + #4}
	\IfSubStr{#6}{t}{%
		\draw[-] (\xright,\ytop) to (\xleft,\ytop); 
	}{}
	\IfSubStr{#6}{l}{%
		\draw[-] (\xleft,\ybot) to (\xleft,\ytop); 
	}{}
	\IfSubStr{#6}{b}{%
		\draw[-] (\xright,\ybot) to (\xleft,\ybot); 
	}{}
	\IfSubStr{#6}{r}{%
		\draw[-] (\xright,\ybot) to (\xright,\ytop); 
	}{}
	\node at (\xmid, \ymid) {\scriptsize #5};
}

\newcommand\arrowgrid{} 
\def\arrowgrid[#1,#2,#3,#4,#5,#6,#7]{%
	\foreach \i in {1,...,#4}
	{
		\pgfmathsetmacro{\ytop}{#2 + \i - 1}
		\pgfmathsetmacro{\ymid}{#2 + \i - 0.5}
		\pgfmathsetmacro{\ybot}{#2 + \i}

		\IfInteger{#5}{%
			\pgfmathtruncatemacro{\iindexcalc}{#5 + \i - 1}%
			\providecommand{\iindex}{\iindexcalc}%
		}{%
			\providecommand{\iindex}{#5}%
		}

		\foreach \j in {1,...,#3}
		{
			\pgfmathsetmacro{\xleft}{#1 + \j - 1}
			\pgfmathsetmacro{\xmid}{#1 + \j - 0.5}
			\pgfmathsetmacro{\xright}{#1 + \j}

			\IfInteger{#6}{%
				\pgfmathtruncatemacro{\jindexcalc}{#6 + \j - 1}%
				\providecommand{\jindex}{\jindexcalc}%
			}{%
				\providecommand{\jindex}{#6}%
			}

			\draw[-] (\xleft,\ybot) to (\xleft,\ytop);
			\draw[-] (\xright,\ytop) to (\xleft,\ytop);
			\node at (\xmid,\ymid) {\scriptsize $#7_{\iindex,\jindex}$}; 
			\ifthenelse{\i=#4}{
				\draw[-] (\xright,\ybot) to (\xleft,\ybot);
			}{}
		}
		\pgfmathtruncatemacro{\xlast}{#1 + #3}
		\draw[-] (\xlast,\ybot) to (\xlast,\ytop);
	}
}

\DeclareMathOperator{\id}{id}

\theoremstyle{plain}
\newtheorem{thm}{Theorem}[section]

\theoremstyle{definition}

\newtheorem{example}[thm]{Example}

\newtheorem{proposition}[thm]{Proposition}

\title{Three explicit examples of Lie $2$-groupoids}
\author{M.~Jotz, L.~Mu\ss m\"acher}
\subjclass{Primary: 
Secondary: 
}

 \email{madeleine.jotz@uni-wuerzburg.de}
\email{linus.mussmaecher@jku.at}

\thanks{}

\thanks{$^*$Corresponding author}

\begin{document}

\begin{abstract}
	This note constructs three explicit classes of Lie $2$-groupoids arising from double Lie groupoids via the bar construction of Artin and Mazur \cite{ArMa66,MeTa11}, following a result of Mehta and Tang \cite{MeTa11}: double Lie groupoids give rise to bisimplicial manifolds whose truncations define Lie $2$-groupoids. This construction is carried out for comma double Lie groupoids \cite{BrMa92}, transitive double Lie groupoids \cite{BrMa92,JoMa21}, and split VB-groupoids \cite{Mackenzie92} described by $2$-term representations up to homotopy \cite{GrMe17}.
	\end{abstract}

\maketitle

\bigskip

	\tableofcontents
	
	\section{Introduction}

Lie groupoids  play a central role in areas such as foliation theory, Poisson geometry, and the theory of differentiable stacks. In recent years, increasing attention has been devoted to \emph{higher} Lie groupoids, which encode symmetries not only between objects but also between morphisms, see e.g.~\cite{Henriques08,Getzler09,Zhu09,TsZh06,BaLa04}. Among these, Lie $2$-groupoids form the first genuinely higher level: they provide geometric models for categorified symmetries.

While the abstract theory of higher Lie groupoids is well developed, the range of explicit geometric constructions remains limited.  The general Lie $2$-groupoid of a $2$-term complex can be found in \cite{HoSt19}, the Lie $2$-groupoid integrating an exact Courant algebroid in \cite{ShZh17,LiSe12,Severa17}, and explicit examples of Lie $2$-groups \cite{BaCr04} are constructed in \cite{GiXu09,BaStCrSc07}.
Simple and concrete examples are particularly valuable because they allow one to test general constructions—most notably differentiation procedures that associate higher Lie algebroids to higher Lie groupoids \cite{LiRyWeZh26,CaHo26}, in analogy with the classical passage from Lie groups to Lie algebras. Carrying out such differentiation methods in specific geometric situations in which the infinitesimal version is already known will help clarify the structures that appear and the role played by the higher simplicial geometry.

The aim of this note is to produce explicit families of Lie $2$-groupoids starting from familiar geometric structures. The construction is based on the observation \cite{MeTa11} that double Lie groupoids naturally give rise to bisimplicial manifolds and so to simplicial manifolds via the bar construction introduced by Artin and Mazur \cite{ArMa66}. The obtained simplicial manifolds are then Lie $2$-groupoids \cite{MeTa11}. Explicitly illustrating this relationship by focusing on simple examples helps clarify how familiar geometric structures can give rise to higher groupoid geometry.

This paper carries out this construction explicitly for three natural classes of double Lie groupoids. The first class consists of comma double Lie groupoids, introduced in work of Brown and Mackenzie \cite{BrMa92}. These double groupoids arise canonically from morphisms of Lie groupoids and encode the geometry of the associated comma constructions. The second class is that of transitive double Lie groupoids, studied by Brown and Mackenzie \cite{BrMa92} and later refined by Jotz and Mackenzie \cite{JoMa21}. The third class arises from VB-groupoids \cite{Pradines88, Mackenzie92}, viewed as special examples of double Lie groupoids two parallel sides of which are given by vector bundles.

In the split case, VB-groupoids are described by $2$-term representations up to homotopy of Lie groupoids, following the celebrated result of Gracia-Saz and Mehta \cite{GrMe17} -- see also \cite{ArCr13} for the notion of representations up to homotopy of Lie groupoids. Interpreting these structures as special double Lie groupoids allows the same bar construction to produce Lie $2$-groupoids encoding the corresponding representation-theoretic data. The same construction applied to VB-groupoids with trivial side or core yields Lie $1$-groupoids, which are defined as the semi-direct product of the Lie groupoid side with a representation. For this reason, the general construction presented here is considered the semi-direct product of a Lie groupoid with a $2$-term representation up to homotopy.

The goal of this paper is not to develop new abstract theory but rather to make these three constructions completely explicit and ready to be used in further work. For each of the three classes, the resulting simplicial manifolds are described and the surmersive but generally not diffeomorphic $2$-horn maps are given explicitly. These examples are intended as concrete models that can be used to explore differentiation procedures for higher Lie groupoids \cite{LiRyWeZh26,CaHo26} and to better understand the geometry underlying such constructions.

Comma double Lie groupoids \cite{BrMa92} are known to differentiate to comma double Lie algebroids \cite{JoMa21}, which are canonically associated to Lie $2$-algebroids that split in a double complex \cite{MePi21}. Similarly, transitive double Lie groupoids \cite{BrMa92} differentiate to transitive double Lie algebroids \cite{JoMa21} and so to \emph{transitive Lie $2$-algebroids} -- these objects are described explicitly by the first author in a work in progress. Finally, VB-groupoids differentiate to VB-algebroids \cite{Pradines88,Mackenzie05,GrMe17}, which define Lie $2$-algebroids as semi-direct products of Lie algebroids with $2$-term representations up to homotopy \cite{ShZh17,Jotz19b}. In all three cases, the bar construction yields at the global level a class of Lie $2$-groupoids, which has to differentiate to the class of Lie $2$-algebroids defined by the infinitesimal level. This will be described in a subsequent work.

  \subsection*{Outline of the paper}
 The paper is organized as follows. After a summary of the necessary background on bisimplicial manifolds, double Lie groupoids and Mehta and Tang's application of the bar construction to these objects \cite{MeTa11}, the three classes of examples are discussed. Section 3 focuses on comma double Lie groupoids, Section 4 on transitive double Lie groupoids, and Section 5 on split VB-groupoids and $2$-term representations up to homotopy.

 \subsection*{Acknowledgement}
The authors warmly thank Chenchang Zhu for interesting discussions, and Karandeep Jandu Singh for his careful reading of the master's thesis of the second author, from which this paper was derived.

	\section{Preliminaries}\label{prelim}
	This section recalls necessary background on Lie $n$-groupoids, double Lie groupoids and the bar construction.
	\subsection{Simplicial manifolds and Lie $n$-groupoids}
	A \textbf{simplicial manifold} $X_\bullet$ consists of
	 a sequence $X_{\bullet} = \{X_r\}_{r \geq 0}$ of smooth manifolds $X_r$,
		smooth surjective submersions $f_i^r\colon X_r \to X_{r-1}$ (for $r\geq 1$) called \textbf{face maps} or simply \textbf{faces} for all $0 \leq i \leq r$,
		 embeddings $\delta_i^r\colon X_r \to X_{r+1}$ (for $r\geq 0$) called \textbf{degeneracy maps}
		for all $0 \leq i \leq r$,
	satisfying the following equations:
	\begin{align*}
		f^{r-1}_i f^r_j &= f^{r-1}_{j-1} f^r_i
		&&\quad \text{ for all } 0 \leq i < j \leq r \\
		\delta^{r+1}_i \delta^r_j &= \delta^{r+1}_{j+1} \delta^r_i
		&&\quad \text{ for all } 0 \leq i \leq j \leq r \\
		f^{r+1}_i \delta^r_j &= \delta^{r-1}_{j-1} f^r_i
		&&\quad \text{ for all } 0 \leq i < j \leq r \\
		f^{r+1}_i \delta^r_j &= \id
		&&\quad \text{ for } i = j \text{ or } i = j+1 \\
		f^{r+1}_i \delta^r_j &= \delta^{r-1}_{j} f^r_{i-1}
		&&\quad \text{ for all } 0 \leq j + 1 < i \leq r \text{.}
	\end{align*}

	Let $X_\bullet$ and $Y_\bullet$ be two simplicial manifolds. A collection $\Phi_\bullet=(\Phi_r)_{r\geq 0}$ of smooth maps $\Phi_r\colon X_r\to Y_r$ for all $r\geq 0$ is a \textbf{morphism of simplicial manifolds} if 
\[\begin{tikzcd}
	{X_{r+1}} && {Y_{r+1}} && {X_r} && {Y_r} \\
	{X_r} && {Y_r} && {X_{r+1}} && {Y_{r+1}}
	\arrow["{\Phi_{r+1}}", from=1-1, to=1-3]
	\arrow["{f^{r+1}_i}"', from=1-1, to=2-1]
	\arrow["{f^{r+1}_i}", from=1-3, to=2-3]
	\arrow["{\Phi_r}", from=1-5, to=1-7]
	\arrow["{\delta^r_i}"', from=1-5, to=2-5]
	\arrow["{\delta^r_i}", from=1-7, to=2-7]
	\arrow["{\Phi_r}"', from=2-1, to=2-3]
	\arrow["{\Phi_{r+1}}"', from=2-5, to=2-7]
\end{tikzcd}\]
commute for all $r\geq 0$ and all applicable $i$ such that the maps are defined.
\medskip

 For $q\geq 1$ and $0\leq k\leq q$, a \textbf{$(q,k)$-horn} of $X$ is a $q$-tuple 
$(x_0, \ldots, x_{k-1}, x_{k+1}, \ldots, x_q)\in (X_{q-1})^q$
such that 
$f_i^{q-1}(x_j)=f_{j-1}^{q-1}(x_i)$
for $0\leq i<j\leq q$, $i\neq k\neq j$.
 For $q\geq 1$ and $0\leq k\leq q$ the smooth manifold of $(q,k)$-horns is written $\wedge_{q,k}X$ and the $k$-th \textbf{horn map} of the $q$-th layer is
	\begin{equation*}
		\lambda_{q,k}\colon X_q \to \wedge_{q,k}X,\qquad  x \mapsto (f_0^q(x), \dots, f_{k-1}^q(x), f_{k+1}^q(x), \dots, f_q^q(x))\text{.}
	\end{equation*}
 The simplicial manifold $X_\bullet$ is called a \textbf{Lie $n$-groupoid} if the horn maps $\lambda_{q,k}$ are surjective submersions at every level and diffeomorphisms for levels $q > n$.
		See \cite{Henriques08, Zhu09}
		 or \cite[Section 4.1]{MeTa11}. A morphism of simplicial manifolds between Lie $n$-groupoids is called a \textbf{morphism of Lie $n$-groupoids}.

\subsection{Bisimplicial manifolds}

The bar construction is applied to \textbf{bisimplicial manifolds}.
	A \textbf{bisimplicial manifold} $X_{\bullet, \bullet}$ consists of 
	\begin{itemize}
		\item a two-parameter sequence $X_{\bullet, \bullet} = \{X_{p,q}\}_{p,q \geq 0}$ of smooth manifolds $X_{p,q}$,
		\item smooth surjective submersions $v^{p,q}_i\colon X_{p,q} \to X_{p,q-1}$ called \textbf{vertical face maps} for all $0 \leq i \leq q$ and all $p \geq 0, q > 0$,
		\item smooth surjective submersions $h^{p,q}_i\colon X_{p,q} \to X_{p-1,q}$ called \textbf{horizontal face maps} for all $0 \leq i \leq p$ and all $p >0, q \geq 0$,
		\item embeddings $\mu^{p,q}_i\colon X_{p,q} \to X_{p,q+1}$ called \textbf{vertical degeneracy maps} for all $0 \leq i \leq q$ and all $p \geq 0, q \geq 0$,
		\item embeddings $\eta^{p,q}_i\colon X_{p,q} \to X_{p+1,q}$ called \textbf{horizontal degeneracy maps} for all $0 \leq i \leq p$ and all $p \geq 0, q \geq 0$,
	\end{itemize}
	such that the following holds:
	\begin{enumerate}
		\item\label{item:bisimplicial-cond1}
		For each fixed $p \geq 0$, the column $X_{p, \bullet}$ with faces $v_i^{p,q}\colon X_{p,q}\to X_{p,q-1}$  for $q\geq 1$ and $0 \leq i \leq q$ and degeneracies $\mu_i^{p,q}\colon X_{p,q}\to X_{p,q+1}$ for $q\geq 0$ and $0 \leq i \leq q$ is a simplicial manifold.
		\item\label{item:bisimplicial-cond2}
		For each fixed $q \geq 0$, the row $X_{\bullet, q}$ with faces $h_i^{p,q}\colon X_{p,q}\to X_{p-1,q}$ for $p\geq 1$ and $0 \leq i \leq p$ and degeneracies $\eta_i^{p,q}\colon X_{p,q}\to X_{p+1,q}$ for $p\geq 0$ and $0\leq i \leq p$ is a simplicial manifold.
		\item\label{item:bisimplicial-cond3}
		The horizontal and vertical structure maps interchange with each other as follows: 
		\begin{align*}
			v^{p-1,q}_j h^{p,q}_i &= h^{p,q-1}_i v^{p,q}_j
			&&\quad \text{ for all } p > 0,q > 0\\
			\mu^{p+1,q}_j \eta^{p,q}_i &= \eta^{p,q+1}_i \mu^{p,q}_j
			&&\quad \text{ for all } p\geq 0,q \geq 0\\
			v^{p+1,q}_j \eta^{p,q}_i &= \eta^{p,q-1}_i v^{p,q}_j
			&&\quad \text{ for all } p \geq 0, q > 0\\
			\mu^{p-1,q}_j h^{p,q}_i &= h^{p,q+1}_i \mu^{p,q}_j
			&&\quad \text{ for all } p > 0, q \geq 0
		\end{align*}
		and for all $0 \leq j \leq q$ and  $0 \leq i \leq p$.
	\end{enumerate}

\subsection{Double Lie groupoids and their bisimplicial nerves}
A double groupoid is a \emph{groupoid object in the category of
  groupoids} (see \cite{BrMa92} and historical remarks therein).  That is, a double groupoid consists of a set $\Gamma$
that has two groupoid structures over two bases $G$ and $H$, which are
themselves groupoids over a base $M$, such that the structure maps of
each groupoid structure on $\Gamma$ are morphisms with respect to the
other. 
\begin{equation*}
		\begin{tikzcd}
	        \Gamma
			\arrow[shift left]{r}{s_G}
			\arrow[shift right, swap]{r}{t_G}
			\arrow[shift left]{d}{s_H}
			\arrow[shift right, swap]{d}{t_H}
			& G
			\arrow[shift left]{d}{s}
			\arrow[shift right, swap]{d}{t}
			\\ H
			\arrow[shift right, swap]{r}{t}
			\arrow[shift left]{r}{s}
			& M
		\end{tikzcd}
	\end{equation*}
The groupoids $G$ and $H$ are the \textbf{side groupoids} and
the set $M$ is the \textbf{double base}. The \textbf{core} $K$ of a double groupoid
$(\Gamma,G,H,M)$ is the set of elements which project under the two
sources of $\Gamma$ to units of $G$ and $H$. It inherits from the
double structure of $\Gamma$ a groupoid structure over $M$, and the
two targets of $\Gamma$ induce two groupoid morphisms
$\partial_G:=\tg_G\arrowvert_K\colon K\to G$ and $\partial_H:=\tg_H\arrowvert_K\colon K\to H$
over the identity on $M$. Elements of the kernel of the morphism
$K\to G$ commute with elements of the kernel of $K\to H$. These two
groupoid morphisms together constitute the \textbf{core diagram} of the
double groupoid \cite{BrMa92}.
   \begin{equation}\label{trans_core_d}
   \begin{tikzcd}
	K & G \\
	H & M
	\arrow[shift left=1, from=1-2, to=2-2]
	\arrow[shift right=1, from=1-2, to=2-2]
	\arrow[shift right=1, from=2-1, to=2-2]
	\arrow[shift left=1, from=2-1, to=2-2]
	\arrow[shift right=1, from=1-1, to=2-2]
	\arrow[shift left=1, from=1-1, to=2-2]
	\arrow["{\partial_G}", from=1-1, to=1-2]
	\arrow["{\partial_H}"', from=1-1, to=2-1]
      \end{tikzcd}\end{equation}

    A \textbf{double Lie groupoid}  is a double groupoid
    $(\Gamma, G, H, M)$ such that all four groupoids $\Gamma\rr G$,
    $\Gamma\rr H$, $G\rr M$ and $H\rr M$ are Lie groupoids and such
    that the double source map
    $(\s_G, \s_H)\colon \Gamma\to G\times_\s H=\{(g,h)\in G\times
    H\mid \s(g)=\s(h)\}$ is a smooth surjective submersion\footnote{Some sources do not require this last condition and consequently call double Lie groupoids satisfying it \emph{strict} or \emph{full}.}. In that
    case, the core diagram is a core diagram of Lie groupoids.  
    
    A
    double Lie groupoid 
    is
    \textbf{transitive} if $\da_G$ and $\da_H$ are both surjective
    submersions and it is \textbf{locally trivial} if $K\rr M$ is in addition locally trivial.  
   Brown and Mackenzie proved in \cite{BrMa92} that
    locally trivial double Lie groupoids are completely determined by
    their core diagrams, and \cite{JoMa21} extended this to the general transitive case.
    
The core diagram of a transitive double Lie groupoid is 
  also \textbf{transitive}. Generally, a diagram of Lie groupoid morphisms as in \eqref{trans_core_d} such that the subgroupoids $\ker(\da_G)$ and $\ker(\da_H)$
  commute in $K$ is called a \textbf{core diagram} and it is \textbf{transitive} if both Lie groupoid morphisms are surjective submersions.

\bigskip 

		The symbol $\cdot$ is used for both the multiplication in $G$ and $H$, but those in $\Gamma$ are written $\cdot_G$ and $\cdot_H$. Similarly, the unit maps originating from $M$ are written $1^H: M \to H$ and $1^G: M \to G$ while the unit maps $\tilde{1}\colon G \to \Gamma$ and $\tilde{1}\colon H \to \Gamma$ are specified by their argument.
	
	As in \cite{MeTa11} and as originally suggested by Mackenzie, elements $\gamma \in \Gamma$ are squares
		\begin{equation*}
			\begin{tikzpicture}[scale=1.8]
				\coordinate (A) at (0,0) ;
				\coordinate (B) at (0,1) ;
				\coordinate (C) at (1,0) ;
				\coordinate (D) at (1,1) ;

				\fill (A) circle (0.2mm);
				\fill (B) circle (0.2mm);
				\fill (C) circle (0.2mm);
				\fill (D) circle (0.2mm);

				\node[below=1mm, left, ] (An) at (A) {$M \ni (t \circ s_G)(\gamma)$};
				\node[above=1mm, left, ] (Bn) at (B) {$M \ni (t \circ t_G)(\gamma)$};
				\node[below=1mm, right,] (Cn) at (C) {$(s \circ s_G)(\gamma) \in M$};
				\node[above=1mm, right,] (Dn) at (D) {$(s \circ t_G)(\gamma) \in M$};

				\draw[->] (A) to node[left ]{$H \ni t_H(\gamma)$} (B);
				\draw[->] (C) to node[right]{$s_H(\gamma) \in H$} (D);
				\draw[->] (C) to node[below]{$s_G(\gamma) \in G$} (A);
				\draw[->] (D) to node[above]{$t_G(\gamma) \in G$} (B);

				\node at (0.5,0.5) {$\gamma$};
			\end{tikzpicture}
		\end{equation*}
		usually simply written as 
		\begin{equation*}
			\begin{tikzpicture}
				\coordinate (A) at (0,0) ;
				\coordinate (B) at (0,1) ;
				\coordinate (C) at (1,0) ;
				\coordinate (D) at (1,1) ;

				\draw[-] (A) to node[left ]{$t_H(\gamma)$} (B);
				\draw[-] (C) to node[right]{$s_H(\gamma)$} (D);
				\draw[-] (C) to node[below]{$s_G(\gamma)$} (A);
				\draw[-] (D) to node[above]{$t_G(\gamma)$} (B);

				\node at (0.5,0.5) {$\gamma$};
			\end{tikzpicture}
		\end{equation*}
		
		\begin{example}\label{ex_pair_gpd}
		Let $G\rr M$ be a Lie groupoid. Then the pair groupoid $G\times G\rr G$ has a double Lie groupoid structure 
\[\begin{tikzcd}
	{G\times G} & {M\times M} \\
	G & M
	\arrow[shift right, from=1-1, to=1-2]
	\arrow[shift left, from=1-1, to=1-2]
	\arrow[shift right, from=1-1, to=2-1]
	\arrow[shift left, from=1-1, to=2-1]
	\arrow[shift right, from=1-2, to=2-2]
	\arrow[shift left, from=1-2, to=2-2]
	\arrow[shift right, from=2-1, to=2-2]
	\arrow[shift left, from=2-1, to=2-2]
\end{tikzcd}\]
with the product groupoid structure on the top edge. The core of this double Lie groupoid is diffeomorphic to $G$ via 
$g\mapsto (g, 1_{\s(g)})$ and the multiplication $(g, 1_{\s(g)})\cdot (h, 1_{\s(h)})=(gh, 1_{\s(h)})$. The core diagram of $G\times G$ is then
\[\begin{tikzcd}
	G & {M\times M} \\
	G & M
	\arrow["{(\tg, \s)}", from=1-1, to=1-2]
	\arrow["{\id_G}"', from=1-1, to=2-1]
	\arrow[shift right, from=1-1, to=2-2]
	\arrow[shift left, from=1-1, to=2-2]
	\arrow[shift right, from=1-2, to=2-2]
	\arrow[shift left, from=1-2, to=2-2]
	\arrow[shift right, from=2-1, to=2-2]
	\arrow[shift left, from=2-1, to=2-2]
\end{tikzcd}\]
and it is transitive if and only if $G\rr M$ is transitive.
The elements of $G\times G$ are visualized as squares
\begin{equation*}
			\begin{tikzpicture}
				\coordinate (A) at (0,0) ;
				\coordinate (B) at (0,1.5) ;
				\coordinate (C) at (1.5,0) ;
				\coordinate (D) at (1.5,1.5) ;

				\draw[-] (A) to node[left ]{$g_1$} (B);
				\draw[-] (C) to node[right]{$g_2$} (D);
				\draw[-] (C) to node[below]{$(\s(g_1), \s(g_2))$} (A);
				\draw[-] (D) to node[above]{$(\tg(g_1), \tg(g_2))$} (B);

				\node at (0.75,0.75) {$(g_1,g_2)$};
			\end{tikzpicture}
		\end{equation*}
with $(g_1,g_2)\in G\times G$.
		\end{example}

		A \textbf{double Lie groupoid morphism} is a 4-tuple of smooth maps $(\Phi\colon \Gamma \to \Gamma', \phi_G\colon G \to G', \phi_H\colon H \to H', \phi_M\colon M \to M')$ such that in
\begin{equation}\label{morphism_dlg}
\begin{tikzcd}
	\Gamma && {\Gamma'} \\
	& H && {H'} \\
	G && {G'} \\
	& M && {M'}
	\arrow["\Phi", from=1-1, to=1-3]
	\arrow[shift right, from=1-1, to=2-2]
	\arrow[shift left, from=1-1, to=2-2]
	\arrow[shift left, from=1-1, to=3-1]
	\arrow[shift right, from=1-1, to=3-1]
	\arrow[shift right, from=1-3, to=2-4]
	\arrow[shift left, from=1-3, to=2-4]
	\arrow[shift right, dashed, from=1-3, to=3-3]
	\arrow[shift left, dashed, from=1-3, to=3-3]
	\arrow["{\phi_H}"{pos=0.3}, from=2-2, to=2-4]
	\arrow[shift right, from=2-2, to=4-2]
	\arrow[shift left, from=2-2, to=4-2]
	\arrow[shift right, from=2-4, to=4-4]
	\arrow[shift left, from=2-4, to=4-4]
	\arrow["{\phi_G}"{pos=0.7}, dashed, from=3-1, to=3-3]
	\arrow[shift right, from=3-1, to=4-2]
	\arrow[shift left, from=3-1, to=4-2]
	\arrow[shift right, dashed, from=3-3, to=4-4]
	\arrow[shift left, dashed, from=3-3, to=4-4]
	\arrow["{\phi_M}"', from=4-2, to=4-4]
\end{tikzcd}
\end{equation}
		the top, bottom, front and back squares are Lie groupoid morphisms.
	If all four maps are bijective (and their inverses are also smooth), the two double groupoids are called \textbf{isomorphic} and $\Phi$ an \textbf{isomorphism}.

\bigskip

Let $(\Gamma, G, H, M)$ be a double Lie groupoid.
Its \textbf{nerve} is the bisimplicial manifold $N_{\bullet,\bullet}\Gamma:=\Gamma^{(\bullet, \bullet)}$ defined by 
\begin{equation*}
	\Gamma^{(p,q)} =
	\left\{
		(\gamma_{i,j})_{\substack{1 \leq i \leq p,\\ 1 \leq j \leq q}} \in \Gamma^{p \times q}
		\left|
		\begin{aligned}
			s_H(\gamma_{i,j}) &= t_H(\gamma_{i,j+1}) && \forall\, 1 \leq i \leq p,\, 1 \leq j < q \\
			s_G(\gamma_{i,j}) &= t_G(\gamma_{i+1,j}) && \forall\, 1 \leq i < p,\, 1 \leq j \leq q
		\end{aligned}
	\right.\right\}.
\end{equation*}
for $p,q\geq 1$,
and $\Gamma^{(0,0)} = M$, as well as $\Gamma^{(0,q)}=G^{(q)}$ and $\Gamma^{(p,0)}=H^{(p)}$ for $p,q\geq 1$.
In other words, an element of $\Gamma^{(p,q)}$ forms a rectangular grid of compatible elements of $\Gamma$
\begin{equation*}
	\begin{tikzpicture}[y=-1cm]
			\arrowsquare[0,0,1,1,$\gamma_{1,1}$,tlbr]
			\arrowsquare[1,0,1,1,$\gamma_{1,2}$,tbr]
			\arrowsquare[0,1,1,1,$\gamma_{2,1}$,lbr]
			\arrowsquare[1,1,1,1,$\gamma_{2,2}$,br]

			\node at (2.5, 1) {$\dots$};
			\node at (1, 2.5) {$\vdots$};
			\node at (2.5, 2.5) {$\ddots$};

			\arrowsquare[3,0,1,1,$\gamma_{1,q}$,tlbr]
			\arrowsquare[3,1,1,1,$\gamma_{2,q}$,lbr]

			\arrowsquare[0,3,1,1,$\gamma_{p,1}$,tlbr]
			\arrowsquare[1,3,1,1,$\gamma_{p,2}$,tbr]

			\node at (2.5,3.5) {$\dots$};
			\node at (3.5,2.5) {$\vdots$};

			\arrowsquare[3,3,1,1,$\gamma_{p,q}$,tlbr]
	\end{tikzpicture}
\end{equation*}

The inner face maps consist of multiplying two adjacent columns or rows together.
Hence applying the vertical face map $v^{p,q}_i$ with $0 < i < q$ to the element above yields the following mapping $\Gamma^{(p,q)} \to \Gamma^{(p,q-1)}$:
\begin{equation*}
	\begin{tikzpicture}[y=-1cm]
		\arrowgrid[0,0,2,2,1,1,\gamma]

		\node at (1, 2.5) {$\vdots$};
		\arrowgrid[0,3,2,1,p,1,\gamma]

		\node at (2.5, 1) {$\dots$};
		\arrowgrid[3,0,1,2,1,q,\gamma]

		\node at (2.5, 2.5) {$\ddots$};
		\node at (2.5, 3.5) {$\dots$};
		\node at (3.5, 2.5) {$\vdots$};
		\arrowgrid[3,3,1,1,p,q,\gamma]

		\node at (5,2) {$\mapsto$};

		\arrowgrid[6,0,2,2,1,1,\gamma]

		\node at (7, 2.5) {$\vdots$};
		\arrowgrid[6,3,2,1,p,1,\gamma]

		\node at (8.5, 1) {$\dots$};
		\arrowsquare[9,0,2,1,$\gamma_{1,i} \cdot \gamma_{1,i+1}$,tlr]
		\arrowsquare[9,1,2,1,$\gamma_{2,i} \cdot \gamma_{2,i+1}$,tlrb]

		\node at (8.5, 2.5) {$\ddots$};
		\node at (8.5, 3.5) {$\dots$};
		\node at (10, 2.5) {$\vdots$};
		\arrowsquare[9,3,2,1,$\gamma_{p,i} \cdot \gamma_{p,i+1}$,tlrb]

		\node at (11.5, 1) {$\dots$};
		\arrowgrid[12,0,1,2,1,q,\gamma]

		\node at (11.5, 2.5) {$\ddots$};
		\node at (11.5, 3.5) {$\dots$};
		\node at (12.5, 2.5) {$\vdots$};
		\arrowgrid[12,3,1,1,p,q,\gamma]
	\end{tikzpicture}\text{.}
\end{equation*}
The face maps $v_0^{p,q}$ and $v_q^{p,q}$ delete the first and last columns, respectively.
The  horizontal face maps $h_i^{p,q}$ for $0 \leq i \leq p$ are defined accordingly.

In particular, the interchange law ensures that vertical and horizontal face maps commute.

The $i$-th degeneracy maps insert an additional row or column of trivial elements after the $i$-th row or column.
The vertical degeneracy map $\mu^{p,q}_i$ maps $(\gamma_{k,l}) \in \Gamma^{(p,q)}$ into $\Gamma^{(p,q+1)}$ as follows:
\begin{equation*}
	\begin{tikzpicture}[y=-1cm]
			\arrowsquare[0,0,1,1,$\gamma_{1,1}$,tlbr]
			\arrowsquare[1,0,1,1,$\gamma_{1,2}$,tbr]
			\arrowsquare[0,1,1,1,$\gamma_{2,1}$,lbr]
			\arrowsquare[1,1,1,1,$\gamma_{2,2}$,br]

			\node at (2.5, 1) {$\dots$};
			\node at (1, 2.5) {$\vdots$};
			\node at (2.5, 2.5) {$\ddots$};

			\arrowsquare[3,0,1,1,$\gamma_{1,q}$,tlbr]
			\arrowsquare[3,1,1,1,$\gamma_{2,q}$,lbr]

			\arrowsquare[0,3,1,1,$\gamma_{p,1}$,tlbr]
			\arrowsquare[1,3,1,1,$\gamma_{p,2}$,tbr]

			\node at (2.5,3.5) {$\dots$};
			\node at (3.5,2.5) {$\dots$};

			\arrowsquare[3,3,1,1,$\gamma_{p,q}$,tlbr]

			\node at (4.5,2) {$\mapsto$};

			\begin{scope}[shift={(-1,0)}]
			\arrowsquare[6,0,1,1,$\gamma_{1,1}$,tlbr]
			\arrowsquare[7,0,1,1,$\gamma_{1,2}$,tbr]
			\arrowsquare[6,1,1,1,$\gamma_{2,1}$,lbr]
			\arrowsquare[7,1,1,1,$\gamma_{2,2}$,br]
			\node at (7, 2.5) {$\vdots$};
			\arrowsquare[6,3,1,1,$\gamma_{p,1}$,tlbr]
			\arrowsquare[7,3,1,1,$\gamma_{p,2}$,tbr]

			\node at (8.5, 1) {$\dots$};
			\node at (8.5, 2.5) {$\ddots$};
			\node at (8.5, 3.5) {$\dots$};

			\arrowsquare[9,0,1,1,$\gamma_{1,i}$,tlbr]
			\arrowsquare[9,1,1,1,$\gamma_{2,i}$,lbr]
			\node at (9.5,2.5) {$\vdots$};
			\arrowsquare[9,3,1,1,$\gamma_{p,i}$,tlbr]

			\arrowsquare[10,0,1,1,$\bar{1}_{1}$,tbr]
			\arrowsquare[10,1,1,1,$\bar{1}_{2}$,br]
			\node at (10.5,2.5) {$\vdots$};
			\arrowsquare[10,3,1,1,$\bar{1}_{p}$,tlbr]

			\arrowsquare[11,0,1,1,$\gamma_{1,i+1}$,tbr]
			\arrowsquare[11,1,1,1,$\gamma_{2,i+1}$,br]
			\node at (11.5,2.5) {$\vdots$};
			\arrowsquare[11,3,1,1,$\gamma_{p,i+1}$,tlbr]

			\node at (12.5, 1) {$\dots$};
			\node at (12.5, 2.5) {$\ddots$};
			\node at (12.5, 3.5) {$\dots$};

			\arrowsquare[13,0,1,1,$\gamma_{1,q}$,tlbr]
			\arrowsquare[13,1,1,1,$\gamma_{2,q}$,lbr]
			\node at (13.5,2.5) {$\vdots$};
			\arrowsquare[13,3,1,1,$\gamma_{p,q}$,tlbr]
			\end{scope}
	\end{tikzpicture}
\end{equation*}
where $\bar{1}_j $ stands for 
$\tilde{1}_{t_H(\gamma_{j,i+1})} = \tilde{1}_{s_H(\gamma_{j,i})}$.
The  horizontal degeneracy maps $\eta_i^{p,q}$ for $0 \leq i \leq p$ are again defined analogously, with rows replacing columns.

It is  easy to see from this description that the bisimplicial identities are satisfied by the horizontal and vertical face and degeneracy maps.

\bigskip

Note that a morphism of double Lie groupoids as in \eqref{morphism_dlg} defines in an obvious manner a morphism of the corresponding bisimplicial nerves.
For all $p,q\geq 0$,  the smooth map $\Phi^{(p,q)}\colon \Gamma^{(p,q)}\to{\Gamma'}^{(p,q)}$ is defined by \[\Phi^{(p,q)}\left((\gamma_{i,j})_{\substack{1 \leq i \leq p,\\ 1 \leq j \leq q}}\right)=\left(\Phi(\gamma_{i,j})\right)_{\substack{1 \leq i \leq p,\\ 1 \leq j \leq q}}
\]
and intertwines the face and degeneracy maps.

\subsection{The bar construction on double Lie groupoids}\label{sec:bar-construction}
The \textbf{bar construction} as first described by Artin and Mazur \cite{ArMa66} is a natural way to obtain a simplicial manifold from a bisimplicial manifold.
In \cite{MeTa11}, Mehta and Tang apply it to a double Lie groupoid $(\Gamma, G, H, M)$ and obtain\footnote{Note that here, the double Lie groupoids are strict, so the bar construction yields Lie $2$-groupoids, see \cite{MeTa11}.} a Lie $2$-groupoid $W_\bullet\Gamma$. For simplicity, only the result of this construction is given here. See \cite{MeTa11} or \cite{ArMa66} for the general bar construction.

For $r=0$, $W_0\Gamma:=M$, while for $r=1$,
\[ W_1\Gamma=\{(g,h)\in G\times H \mid \s(g)=\tg(h)\}.
\]
For $r\geq 2$, 
\[ W_r\Gamma=\left\{\begin{array}{c}
\left(g, \left(\gamma_{i,j}\right)_{1\leq i\leq j< r}, h\right)\\
\in G\times \Gamma^{\frac{r(r-1)}{2}}\times H 
\end{array}\left| \begin{array}{c}
\s(g)=\tg(\tg_G(\gamma_{1,1}))\\
\s_H(\gamma_{i,j})=\tg_H(\gamma_{i,(j+1)}) \text{ for } 1\leq i< r\text{ and } i\leq j<r-1 \\
\s_G(\gamma_{i,j})=\tg_G(\gamma_{(i+1),j}) \text{ for } 1\leq j< r\text{ and } 1\leq i<j\\
\tg(h)=\s(\s_H(\gamma_{r-1,r-1}))
\end{array}\right.\right\}
\]
\begin{equation}\label{eq:double-groupoid-w-r}
	\begin{tikzpicture}[y=-1cm]
		\node at (0.5,-0.4) {$g$};
		\draw[->] (1,0) to (0,0);

		\arrowsquare[1,0,1,1,$\gamma_{1,1}$,tlbr]

		\arrowsquare[2,0,1,1,$\gamma_{1,2}$,tbr]
		\arrowsquare[2,1,1,1,$\gamma_{2,2}$,lbr]

		\node at (3.5,1) {$\dots$};
		\arrowsquare[4,0,1,1,$\gamma_{1,r-1}$,tlr]
		\arrowsquare[4,1,1,1,$\gamma_{2,r-1}$,tlbr]

		\node at (3.5,2.5) {$\ddots$};
		\node at (4.5,2.5) {$\vdots$};
		\arrowsquare[4,3,1,1,$\gamma_{r-1,r-1}$,tlbr]
	
		\node[text width=1cm] at (5.6, 4.5) {$h$};
		\draw[->] (5,5) to (5,4);
	\end{tikzpicture}
\end{equation}

The face map $f^r_i\colon W_r\Gamma \to W_{r-1}\Gamma$ for $0 < i < r$ sends the element \eqref{eq:double-groupoid-w-r} above to
\begin{equation*}
	\begin{tikzpicture}[y=-1.2cm, x=1.2cm, scale=1]
		\node at (0.5, -0.4) {$g$};

		\draw[->] (1,0) to (0,0);

		\arrowsquare[1,0,1,1,$\gamma_{1,1}$,tlbr]
		\arrowsquare[2,0,1,1,$\gamma_{1,2}$,tlbr]
		\arrowsquare[2,1,1,1,$\gamma_{2,2}$,tlbr]

		\node at (3.5,0.5) {$\dots$};
		\node at (3.5,1.5) {$\ddots$};

		\arrowsquare[4,0,2,1,$\gamma_{1,i-1} \cdot \gamma_{1,i}$,tlbr]
		\node at (5,1.5) {$\vdots$};
		\arrowsquare[4,2,2,1,$\gamma_{i-1,i-1} \cdot \gamma_{i-1,i}$,tlbr]

		\draw[dotted] (5,4) to (5,3);
		\draw[dotted] (6,4) to (5,4);

		\arrowsquare[6,0,1,1,$\gamma_{1,i+1}$,tbr]
		\node at (6.5, 1.5) {$\vdots$};
		\arrowsquare[6,2,1,1,$\gamma_{i-1,i+1}$,tbr]
		\arrowsquare[6,3,1,2,
		$\substack{\gamma_{i,i+1} \\ \cdot \\ \gamma_{i+1,i+1}}$
		,lbr]

		\node at (7.5,0.5) {$\dots$};
		\node at (7.5,2.5) {$\dots$};
		\node at (7.5,4.0) {$\dots$};
		\node at (7.5,5.5) {$\ddots$};

		\arrowsquare[8,0,1,1,$\gamma_{1,r-1}$,tlbr]
		\node at (8.5,1.5) {$\vdots$};
		\arrowsquare[8,2,1,1,$\gamma_{i-1,r-1}$,tlbr]
		\arrowsquare[8,3,1,2,
		$\substack{\gamma_{i,r-1} \\ \cdot \\ \gamma_{i+1,r-1}}$
		,lbr]
		\node at (8.5,5.5) {$\vdots$};
		\arrowsquare[8,6,1,1,$\gamma_{r-1,r-1}$,tlbr]
		\draw[->] (9,8) to (9,7);
		\node at (9.3,7.5) {$h$};
	\end{tikzpicture}
\end{equation*}
in $W_{r-1}\Gamma$, where for $i=1$, $\gamma_{1,0}\cdot \gamma_{1,1}$ is replaced by $g\cdot \tg_G(\gamma_{1,1})$ and for $i=r-1$
the product $\gamma_{r-1,r-1}\cdot \gamma_{r,r-1}$ stands for $\s_H(\gamma_{r-1,r-1})\cdot h$.
The $0$-th and $r$-th face maps simply delete the first row and last column respectively, i.e.~the basic simplex from \eqref{eq:double-groupoid-w-r} has the following images under these maps.
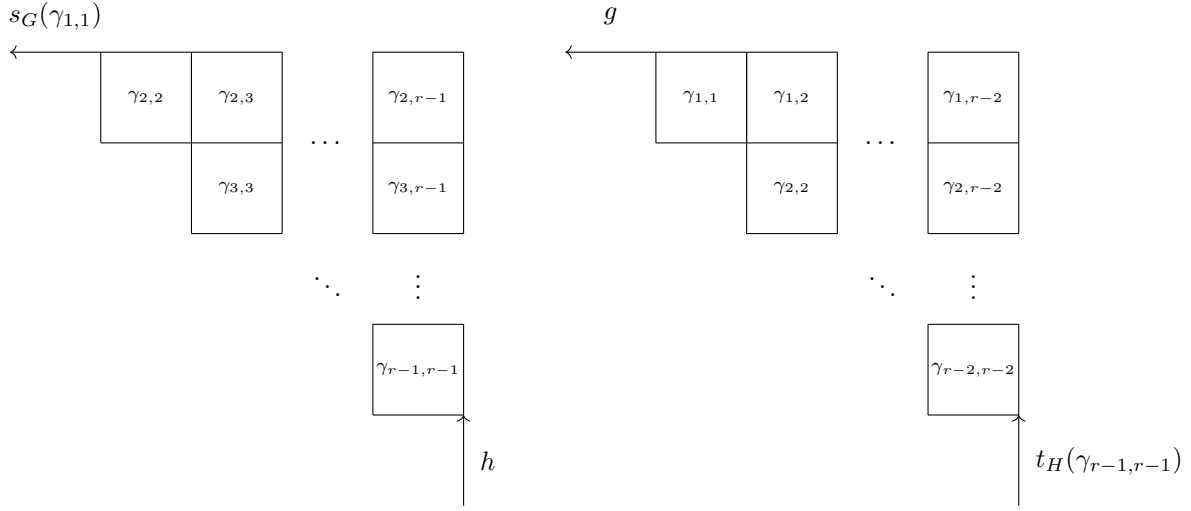
\begin{equation*}
	\begin{tikzpicture}[y=-1.2cm,x=1.2cm]
		\node at (0.5,-0.4) {$s_G(\gamma_{1,1})$};
		\draw[->] (1,0) to (0,0);

		\arrowsquare[1,0,1,1,$\gamma_{2,2}$,tlbr]

		\arrowsquare[2,0,1,1,$\gamma_{2,3}$,tlr]
		\arrowsquare[2,1,1,1,$\gamma_{3,3}$,tlbr]

		\node at (3.5,1) {$\dots$};
		\arrowsquare[4,0,1,1,$\gamma_{2,r-1}$,tlr]
		\arrowsquare[4,1,1,1,$\gamma_{3,r-1}$,tlbr]

		\node at (3.5,2.5) {$\ddots$};
		\node at (4.5,2.5) {$\vdots$};
		\arrowsquare[4,3,1,1,$\gamma_{r-1,r-1}$,tlbr]
	
		\node[text width=1cm] at (5.6, 4.5) {$h$};
		\draw[->] (5,5) to (5,4);
	\end{tikzpicture}
	\begin{tikzpicture}[y=-1.2cm,x=1.2cm]
		\node at (0.5,-0.4) {$g$};
		\draw[->] (1,0) to (0,0);

		\arrowsquare[1,0,1,1,$\gamma_{1,1}$,tlbr]

		\arrowsquare[2,0,1,1,$\gamma_{1,2}$,tlr]
		\arrowsquare[2,1,1,1,$\gamma_{2,2}$,tlbr]

		\node at (3.5,1) {$\dots$};
		\arrowsquare[4,0,1,1,$\gamma_{1,r-2}$,tlr]
		\arrowsquare[4,1,1,1,$\gamma_{2,r-2}$,tlbr]

		\node at (3.5,2.5) {$\ddots$};
		\node at (4.5,2.5) {$\vdots$};
		\arrowsquare[4,3,1,1,$\gamma_{r-2,r-2}$,tlbr]
	
		\node at (6.0, 4.5) {$t_H(\gamma_{r-1,r-1})$};
		\draw[->] (5,5) to (5,4);
	\end{tikzpicture}
\end{equation*}
In particular, the $1$-st level maps $f^1_0$ and $f^1_1$  map the $1$-simplex
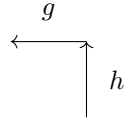
\begin{equation*}
	\begin{tikzpicture}[y=-1cm]
		\node at (0.5,-0.4) {$g$};
		\draw[->] (1,0) to (0,0);
		\node at (1.4,0.5) {$h$};
		\draw[->] (1,1) to (1,0);
	\end{tikzpicture}
\end{equation*}
to $\s(h)$ and $\tg(g)$ respectively.
The image of the basic $2$-simplex 
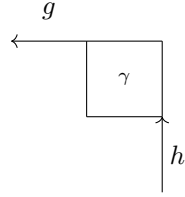
\begin{equation*}
	\begin{tikzpicture}[y=-1cm]
		\node at (0.5,-0.4) {$g$};
		\draw[->] (1,0) to (0,0);
		\arrowsquare[1,0,1,1,$\gamma$,tlbr]
		\node[text width=0.8cm] at (2.5, 1.5) {$h$};
		\draw[->] (2,2) to (2,1);
	\end{tikzpicture}
\end{equation*}
in $W_2\Gamma$ can be visualized by the following diagrams:
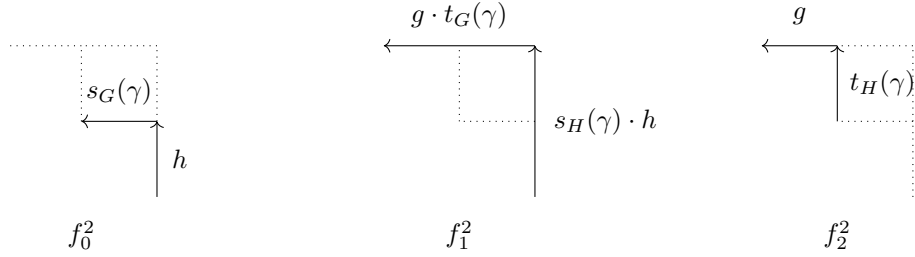
\begin{equation*}
	\begin{tikzpicture}[y=-1cm]
	\begin{scope}[shift={(0,0)}]
		\draw[dotted] (1,0) to (0,0);
		\draw[dotted] (2,0) to (1,0);
		\draw[dotted] (1,1) to (1,0);
		\draw[dotted] (2,1) to (2,0);
		\node at (1.5,0.6) {$s_G(\gamma)$};
		\draw[->] (2,1) to (1,1);
		\node at (2.3,1.5) {$h$};
		\draw[->] (2,2) to (2,1);
		\node at (1, 2.5) {$f_0^2$};
	\end{scope}
	\begin{scope}[shift={(5,0)}]
		\node at (1,-0.4) {$g \cdot t_G(\gamma)$};
		\draw[->] (2,0) to (0,0);
		\draw[dotted] (1,1) to (1,0);
		\draw[dotted] (2,1) to (1,1);
		\node[text width=1.5cm] at (3,1) {$s_H(\gamma) \cdot h$};
		\draw[->] (2,2) to (2,0);
		\node at (1, 2.5) {$f_1^2$};
	\end{scope}
	\begin{scope}[shift={(10,0)}]
		\node at (0.5, -0.4) {$g$};
		\draw[->] (1,0) to (0,0);
		\draw[dotted] (2,0) to (1,0);
		\node at (1.6, 0.5) {$t_H(\gamma)$};
		\draw[->] (1,1) to (1,0);
		\draw[dotted] (2,1) to (2,0);
		\draw[dotted] (2,1) to (1,1);
		\draw[dotted] (2,2) to (2,1);
		\node at (1, 2.5) {$f_2^2$};
	\end{scope}
	\end{tikzpicture}
\end{equation*}
see \cite[665]{MeTa11}.
The maps $f^3_0, f^3_1, f^3_2, f^3_3$ map the basic $3$-simplex
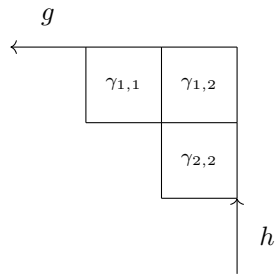
\begin{equation*}
	\begin{tikzpicture}[y=-1cm]
		\draw[->] (1,0) to (0,0);
		\node at (0.5, -0.4) {$g$};
		\arrowsquare[1,0,1,1,$\gamma_{1,1}$,tlbr]
		\arrowsquare[2,0,1,1,$\gamma_{1,2}$,tbr]
		\arrowsquare[2,1,1,1,$\gamma_{2,2}$,lbr]
		\draw[->] (3,3) to (3,2);
		\node at (3.4, 2.5) {$h$};
	\end{tikzpicture}
\end{equation*}
in $W_3\Gamma$ to
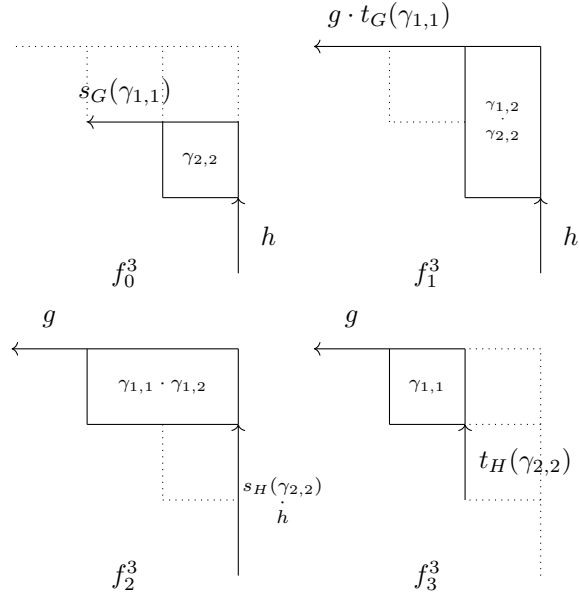
\begin{equation*}
	\begin{tikzpicture}[y=-1cm]
		\begin{scope}
			\draw[dotted] (1,0) to (0,0);
			\draw[dotted] (2,0) to (1,0);
			\draw[dotted] (3,0) to (2,0);
			\draw[dotted] (1,1) to (1,0);
			\draw[dotted] (2,1) to (2,0);
			\draw[dotted] (3,1) to (3,0);
			\draw[->] (2,1) to (1,1);
			\node at (1.5, 0.6) {$s_G(\gamma_{1,1})$};
			\arrowsquare[2,1,1,1,$\gamma_{2,2}$,tlbr]
			\draw[->] (3,3) to (3,2);
			\node at (3.4, 2.5) {$h$};

			\node at (1.5, 3) {$f^3_0$};
		\end{scope}

		\begin{scope}[shift={(4,0)}]
			\draw[->] (2,0) to (0,0);
			\node at (1, -0.4) {$g \cdot t_G(\gamma_{1,1})$};
			\draw[dotted] (1,1) to (1,0);
			\draw[dotted] (2,1) to (1,1);
			\arrowsquare[2,0,1,2,$\substack{\gamma_{1,2} \\ \cdot \\ \gamma_{2,2}}$,tlbr]
			\draw[->] (3,3) to (3,2);
			\node at (3.4, 2.5) {$h$};

			\node at (1.5, 3) {$f^3_1$};
		\end{scope}

		\begin{scope}[shift={(0,4)}]
			\draw[->] (1,0) to (0,0);
			\node at (0.5, -0.4) {$g$};
			\arrowsquare[1,0,2,1,$\gamma_{1,1} \cdot \gamma_{1,2}$,tlbr]
			\draw[dotted] (2,1) to (2,2);
			\draw[dotted] (2,2) to (3,2);
			\draw[->] (3,3) to (3,1);
			\node at (3.5, 2) {$\,\,\,\substack{s_H(\gamma_{2,2}) \\ \cdot \\ h}$};

			\node at (1.5, 3) {$f^3_2$};
		\end{scope}
		\begin{scope}[shift={(4,4)}]
			\draw[->] (1,0) to (0,0);
			\node at (0.5, -0.4) {$g$};
			\arrowsquare[1,0,1,1,$\gamma_{1,1}$,tlbr]
			\draw[->] (2,2) to (2,1);
			\node at (2.8,1.5) {$t_H(\gamma_{2,2})$};
			\draw[dotted] (3,1) to (3,0);
			\draw[dotted] (3,2) to (3,1);
			\draw[dotted] (3,3) to (3,2);
			\draw[dotted] (3,0) to (2,0);
			\draw[dotted] (3,1) to (2,1);
			\draw[dotted] (3,2) to (2,2);

			\node at (1.5, 3) {$f^3_3$};
		\end{scope}
	\end{tikzpicture}
\end{equation*}

The image of the basic $r$-simplex as  above under the degeneracy map $\delta^r_i$ for $0 < i < r$ is constructed by inserting a new unit column between the $(i-1)$-th and $i$-th column, and a new unit row between the $i$-th and $(i+1)$-th row:
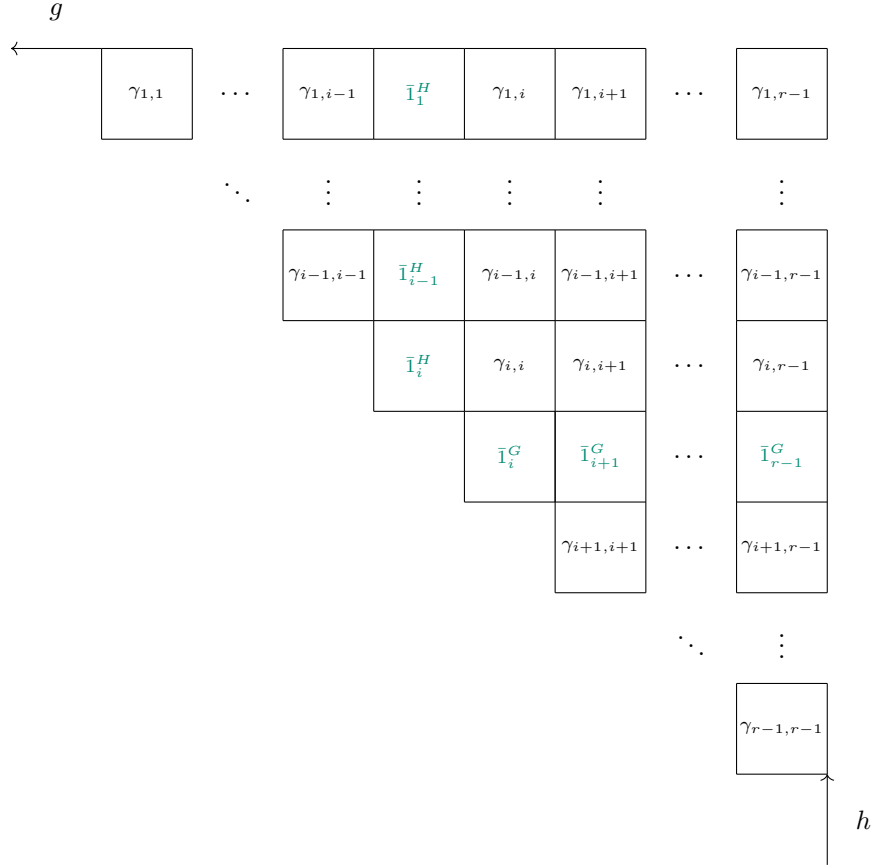
\begin{equation*}
	\begin{tikzpicture}[y=-1.2cm, x=1.2cm, scale=1]
		\node at (0.5, -0.4) {$g$};

		\draw[->] (1,0) to (0,0);

		\arrowsquare[1,0,1,1,$\gamma_{1,1}$,tlbr]

		\node at (2.5,0.5) {$\dots$};
		\node at (2.5,1.5) {$\ddots$};

		\arrowsquare[3,0,1,1,$\gamma_{1,i-1}$, tlbr]
		\node at (3.5,1.5) {$\vdots$};
		\arrowsquare[3,2,1,1,$\gamma_{i-1,i-1}$, tlbr]

		\arrowsquare[4,0,1,1,{\color{PineGreen}$\bar{1}^H_1$},tbr]
		\node at (4.5,1.5) {$\vdots$};
		\arrowsquare[4,2,1,1,{\color{PineGreen}$\bar{1}^H_{i-1}$},tbr]
		\arrowsquare[4,3,1,1,{\color{PineGreen}$\bar{1}^H_i$},lbr]

		\arrowsquare[5,0,1,1,$\gamma_{1,i}$,tbr]
		\node at (5.5,1.5) {$\vdots$};
		\arrowsquare[5,2,1,1,$\gamma_{i-1,i}$,tbr]
		\arrowsquare[5,3,1,1,$\gamma_{i,i}$,br]
		\arrowsquare[5,4,1,1,{\color{PineGreen}$\bar{1}^G_i$},lbr]

		\arrowsquare[6,0,1,1,$\gamma_{1,i+1}$,tbr]
		\node at (6.5,1.5) {$\vdots$};
		\arrowsquare[6,2,1,1,$\gamma_{i-1,i+1}$,tbr]
		\arrowsquare[6,3,1,1,$\gamma_{i,i+1}$,br]
		\arrowsquare[6,4,1,1,{\color{PineGreen}$\bar{1}^G_{i+1}$},lbr]
		\arrowsquare[6,5,1,1,$\gamma_{i+1,i+1}$,lbr]

		\node at (7.5,0.5) {$\dots$};
		\node at (7.5,2.5) {$\dots$};
		\node at (7.5,3.5) {$\dots$};
		\node at (7.5,4.5) {$\dots$};
		\node at (7.5,5.5) {$\dots$};
		\node at (7.5,6.5) {$\ddots$};

		\arrowsquare[8,0,1,1,$\gamma_{1,r-1}$,tlbr]
		\node at (8.5,1.5) {$\vdots$};
		\arrowsquare[8,2,1,1,$\gamma_{i-1,r-1}$,tlbr]
		\arrowsquare[8,3,1,1,$\gamma_{i,r-1}$,lbr]
		\arrowsquare[8,4,1,1,{\color{PineGreen}$\bar{1}^G_{r-1}$},lbr]
		\arrowsquare[8,5,1,1,$\gamma_{i+1,r-1}$,lbr]
		\node at (8.5,6.5) {$\vdots$};
		\arrowsquare[8,7,1,1,$\gamma_{r-1,r-1}$,tlbr]
		\draw[->] (9,9) to (9,8);

		\node at (9.4, 8.5) {$h$};
	\end{tikzpicture}
\end{equation*}
in $W_{r+1}\Gamma$ with
	$\bar{1}^H_j = \tilde{1}(t_H(\gamma_{j,i}))$ for all $1 \leq j \leq i$
and 	$\bar{1}^G_j = \tilde{1}(t_G(\gamma_{i,j}))$ for all $i \leq j < r$.
The $0$-th degeneracy map $\delta^r_0$ sends the standard simplex \eqref{eq:double-groupoid-w-r} to 
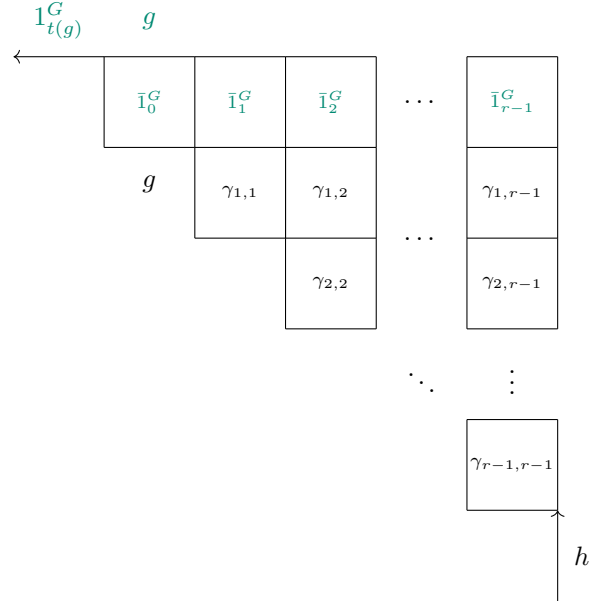
\begin{equation*}
	\begin{tikzpicture}[y=-1.2cm, x=1.2cm, scale=1]
		\node at (0.5,-0.4) {\color{PineGreen}$1^G_{t(g)}$};
		\draw[->] (1,0) to (0,0);

		\arrowsquare[1,0,1,1,\color{PineGreen}$\bar{1}^G_0$,tlbr]
		\node at (1.5,1.4) {$g$};
		\node at (1.5,-0.4) {\color{PineGreen}$g$};
		\arrowsquare[2,0,1,1,\color{PineGreen}$\bar{1}^G_1$,tr]
		\arrowsquare[3,0,1,1,\color{PineGreen}$\bar{1}^G_2$,tr]
		\node at (4.5,0.5) {$\dots$};
		\arrowsquare[5,0,1,1,\color{PineGreen}$\bar{1}^G_{r-1}$,tlr]

		\arrowsquare[2,1,1,1,$\gamma_{1,1}$,tlbr]

		\arrowsquare[3,1,1,1,$\gamma_{1,2}$,tlr]
		\arrowsquare[3,2,1,1,$\gamma_{2,2}$,tlbr]

		\node at (4.5,2) {$\dots$};
		\arrowsquare[5,1,1,1,$\gamma_{1,r-1}$,tlr]
		\arrowsquare[5,2,1,1,$\gamma_{2,r-1}$,tlbr]

		\node at (4.5,3.5) {$\ddots$};
		\node at (5.5,3.5) {$\vdots$};
		\arrowsquare[5,4,1,1,$\gamma_{r-1,r-1}$,tlbr]
	
		\node[text width=1cm] at (6.6, 5.5) {$h$};
		\draw[->] (6,6) to (6,5);
	\end{tikzpicture}
\end{equation*}
with
\begin{equation*}
	\bar{1}^G_0 = \tilde{1}(g)
\end{equation*}
and $\bar{1}^G_j$ for $0 < j < r$ as above.
Meanwhile, the last degeneracy map $\delta^r_r$ maps \eqref{eq:double-groupoid-w-r} to
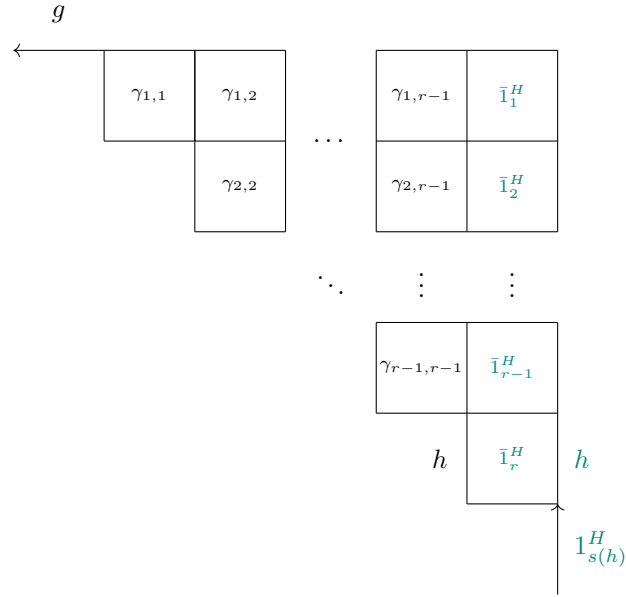
\begin{equation*}
	\begin{tikzpicture}[y=-1.2cm,x=1.2cm]
		\node at (0.5,-0.4) {$g$};
		\draw[->] (1,0) to (0,0);

		\arrowsquare[1,0,1,1,$\gamma_{1,1}$,tlbr]

		\arrowsquare[2,0,1,1,$\gamma_{1,2}$,tlr]
		\arrowsquare[2,1,1,1,$\gamma_{2,2}$,tlbr]

		\node at (3.5,1) {$\dots$};
		\arrowsquare[4,0,1,1,$\gamma_{1,r-1}$,tlr]
		\arrowsquare[4,1,1,1,$\gamma_{2,r-1}$,tlbr]

		\arrowsquare[5,0,1,1,\color{PineGreen}$\bar{1}^H_1$,tbr]
		\arrowsquare[5,1,1,1,\color{PineGreen}$\bar{1}^H_2$,br]
		\node at (5.5,2.5) {$\vdots$};
		\arrowsquare[5,3,1,1,\color{PineGreen}$\bar{1}^H_{r-1}$,tbr]
		\arrowsquare[5,4,1,1,\color{PineGreen}$\bar{1}^H_{r}$,lbr]

		\node at (3.5,2.5) {$\ddots$};
		\node at (4.5,2.5) {$\vdots$};
		\arrowsquare[4,3,1,1,$\gamma_{r-1,r-1}$,tlbr]
		\node at (4.7,4.5) {$h$};
	
		\node[text width=1cm] at (6.6, 4.5) {\color{PineGreen}$h$};
		\node[text width=1cm] at (6.6, 5.5) {\color{PineGreen}$1^H_{s(h)}$};
		\draw[->] (6,6) to (6,5);
	\end{tikzpicture}
\end{equation*}
with
\begin{align*}
	\bar{1}^H_r = \tilde{1}(h)
\end{align*}
and $\bar{1}^H_j$ for $0 < j < r$ as before.

The following proposition is easy to prove. 
\begin{proposition}
Let $(\Gamma, G, H, M)$ and $(\Gamma',G',H',M')$ be double Lie groupoids and consider a morphism $(\Phi\colon \Gamma \to \Gamma', \phi_G\colon G \to G', \phi_H\colon H \to H', \phi_M\colon M \to M')$ of double Lie groupoids as in \eqref{morphism_dlg}. Then $(\Phi, \phi_G, \phi_H, \phi_M)$ defines a morphism 
\[\Phi_\bullet\colon W_\bullet\Gamma\to W_\bullet\Gamma'
\]
of simplicial manifolds. The map $\Phi_0\colon W_0\Gamma=M\to W_0\Gamma'=M'$ equals $\phi_M$ and for $r\geq 1$ the map $\Phi_r\colon W_r\Gamma\to W_r\Gamma'$ is defined by 
\[ \Phi_r\left(g, \left(\gamma_{i,j}\right)_{1\leq i\leq j< r}, h\right)=\left(\phi_G(g), \left(\Phi(\gamma_{i,j})\right)_{1\leq i\leq j< r}, \phi_H(h)\right)
\]
for all $\left(g, \left(\gamma_{i,j}\right)_{1\leq i\leq j< r}, h\right)\in W_r\Gamma$.
\end{proposition}

\begin{example}{The bar construction applied to the pair groupoid of a Lie groupoid}\label{ex_pair_gpd1}
Consider a Lie groupoid $G\rr M$ and the induced double Lie groupoid $G\times G$ as in Example \ref{ex_pair_gpd}.
Here, $W_0(G\times G)=M$ and for $r\geq 1$, the space $W_r(G\times G)$ is the smooth manifold
\[ W_r(G\times G)=G^{(0)}\times G^{(1)}\times \ldots\times G^{(r)}.
\] 
The lower face maps are 
\[ f^1_1\colon M\times G \to M, \quad (x,g)\to x, \qquad f^1_0\colon M\times G\to M, \quad (x,g)\to \s(g),
\]
for $i=0,1,2$
\[ f^2_i\colon M\times G\times G^{(2)} \to M\times G\]
\[f_0^2\colon \left(x,g_1^1,\begin{pmatrix}g^2_1\\ g^2_2\end{pmatrix}\right)\to (\s(g_1^1), g_2^2),
\]
\[f_1^2\colon \left(x,g_1^1,\begin{pmatrix}g^2_1\\ g^2_2\end{pmatrix}\right)\to(x,g_1^2\cdot g_2^2)
\]
\[f_2^2\colon \left(x,g_1^1,\begin{pmatrix}g^2_1\\ g^2_2\end{pmatrix}\right)\to (x, g_1^1),
\]
and for $i=0,1,2,3$
\[ f^3_i\colon M\times G\times G^{(2)}\times G^{(3)} \to M\times G\times G^{(2)}\]
\[f_0^3\colon \left(x,g_1^1,\begin{pmatrix}g^2_1\\ g^2_2\end{pmatrix},\begin{pmatrix}g^3_1\\ g^3_2\\ g^3_3\end{pmatrix}\right)\to \left(\s(g_1^1), g_2^2, \begin{pmatrix} g_2^3\\ g_3^3\end{pmatrix}\right),
\]
\[f_1^3\colon \left(x,g_1^1,\begin{pmatrix}g^2_1\\ g^2_2\end{pmatrix},\begin{pmatrix}g^3_1\\ g^3_2\\ g^3_3\end{pmatrix}\right)\to\left(x,g_1^2\cdot g_2^2,\begin{pmatrix} g_1^3\cdot g_2^3\\ g_3^3\end{pmatrix}\right)
\]
\[f_2^3\colon \left(x,g_1^1,\begin{pmatrix}g^2_1\\ g^2_2\end{pmatrix},\begin{pmatrix}g^3_1\\ g^3_2\\ g^3_3\end{pmatrix}\right)\to \left(x,g_1^1,\begin{pmatrix} g_1^3\\ g_2^3\cdot g_3^3\end{pmatrix}\right)
\]
and 
\[f_3^3\colon \left(x,g_1^1,\begin{pmatrix}g^2_1\\ g^2_2\end{pmatrix},\begin{pmatrix}g^3_1\\ g^3_2\\ g^3_3\end{pmatrix}\right)\to \left(x,g_1^1,\begin{pmatrix} g_1^2\\ g_2^2\end{pmatrix}\right).
\]
It is easy to see that the $3$-horn maps are diffeomorphisms, while the $2$-horn maps
\begin{align*}
	\lambda_{2,0}&=(f^2_1, f^2_2)\colon \left(x,g_1^1,\begin{pmatrix}g^2_1\\ g^2_2\end{pmatrix}\right)\mapsto
	\left(x, g_1^2\cdot g_2^2, x, g_1^1
	\right),\\
	\lambda_{2,1}&=(f^2_0, f^2_2)\colon \left(x,g_1^1,\begin{pmatrix}g^2_1\\ g^2_2\end{pmatrix}\right)\mapsto
	\left(\s(g_1^1), g_2^2, x, g_1^1
	\right), \\
\text{ and }\quad  	\lambda_{2,2}&=(f^2_0, f^2_1)\colon\left(x,g_1^1,\begin{pmatrix}g^2_1\\ g^2_2\end{pmatrix}\right)\mapsto
	\left(\s(g_1^1), g_2^2, x, g_1^2\cdot g_2^2
	\right)
\end{align*}
are surjective submersions.

In general for $r\geq 3$ and $0<i<r$ 
\[ f^r_i\colon G^{(0)}\times G^{(1)}\times \ldots\times G^{(r)}\to G^{(0)}\times G^{(1)}\times \ldots\times G^{(r-1)}, 
\]
sends 
\[ \left( x, g_1^1, (g_1^2, g_2^2), \ldots, (g_j^r)_{j=1,\ldots, r}\right)\]
to \[\left(x, \ldots, (g_j^{i-1})_{1\leq j\leq i-1}, (g_1^{i+1}, \ldots, g_{i-1}^{i+1}, g_i^{i+1}\cdot g_{i+1}^{i+1}) ,\ldots, (g_1^{r}, \ldots, g_{i-1}^{r}, g_i^{r}\cdot g_{i+1}^{r}, g^r_{i+1}, \ldots, g^r_r)\right),
\]
while 
\[ f^r_r\colon G^{(0)}\times G^{(1)}\times \ldots\times G^{(r)}\to G^{(0)}\times G^{(1)}\times \ldots\times G^{(r-1)}
\]
does 
\[\left( x, g_1^1, (g_1^2, g_2^2), \ldots, (g_j^r)_{j=1,\ldots, r}\right)\mapsto \left( x, g_1^1, (g_1^2, g_2^2), \ldots, (g_j^{r-1})_{j=1,\ldots, r-1}\right)
\]
and 
\[ f^r_0\colon G^{(0)}\times G^{(1)}\times \ldots\times G^{(r)}\to G^{(0)}\times G^{(1)}\times \ldots\times G^{(r-1)}
\]
does 
\[\left( x, g_1^1, (g_1^2, g_2^2), \ldots, (g_j^r)_{j=1,\ldots, r}\right)\mapsto \left( \s(g_1^1), g_2^2, \ldots, (g_j^{r})_{j=2,\ldots, r}\right)\text{.}
\]
The degeneracy maps are omitted here and can be derived from the general formula in Section~\ref{sec:bar-construction}.
\end{example}

\section{The bar construction applied to comma double Lie groupoids}
This section briefly recalls the notion of comma double Lie groupoids and describes in detail the \emph{comma Lie $2$-groupoids} produced by applying the bar construction to this class of examples.
\subsection{Comma double Lie groupoids}
Consider two Lie groupoids $K\rr M$ and $G\rr M$ over a
common base $M$ and a Lie groupoid morphism $\Phi\colon K\to G$
fixing $M$. 
The double Lie groupoid $\Lambda:=(\Phi,\Phi)$ with sides $K$ and $G$ and with core
$K$ is defined as follows \cite{BrMa92}.  Its elements are triples
$(k_\tg,g,k_\s)\in K\times G\times K$ such that $\s(k_\tg)=\tg(g)$ and
$\s(k_\s)=\s(g)$. The source and target maps
$\Lambda\to G$ are the maps $(k_\tg,g,k_\s)\mapsto g$ and
$(k_\tg,g,k_\s)\mapsto \Phi(k_\tg)\cdot g\cdot \Phi(k_\s^{-1})$, respectively.
The source and target maps $\Lambda\to K$ are the maps
$(k_\tg,g,k_\s)\mapsto k_\s$ and $(k_\tg,g,k_\s)\mapsto k_\tg$, respectively.  The
composition over $G$ is
\[(k_\tg',\Phi(k_\tg)\cdot g\cdot \Phi(k_\s^{-1}),k_\s')\cdot_G(k_\tg,g,k_\s)=(k_\tg'k_\tg, g, k_\s'k_\s)
\]
and the composition over $K$ is
\[ (k_3,g', k_2)\cdot_K(k_2,g,k_1)=(k_3,g'g,k_1).
\]
The elements of $\Lambda$ are pictured 
\begin{equation*}
			\begin{tikzpicture}
				\coordinate (A) at (0,0) ;
				\coordinate (B) at (0,1) ;
				\coordinate (C) at (1,0) ;
				\coordinate (D) at (1,1) ;

				\draw[-] (A) to node[left ]{$k_\tg$} (B);
				\draw[-] (C) to node[right]{$k_\s$} (D);
				\draw[-] (C) to node[below]{$g$} (A);
				\draw[-] (D) to node[above]{} (B);
\node at (0.5,0.5) {};
			\end{tikzpicture}
		\end{equation*}
		with top edge $\Phi(k_\tg)\cdot g\cdot \Phi(k_\s^{-1})$. For the bar construction later on, however, elements of $\Lambda$ are understood as squares 
		\begin{equation*}
			\begin{tikzpicture}
				\coordinate (A) at (0,0) ;
				\coordinate (B) at (0,1) ;
				\coordinate (C) at (1,0) ;
				\coordinate (D) at (1,1) ;

				\draw[-] (A) to node[left ]{$k_\tg$} (B);
				\draw[-] (C) to node[right]{$k_\s$} (D);
				\draw[-] (C) to node[below]{} (A);
				\draw[-] (D) to node[above]{$g$} (B);
\node at (0.5,0.5) {};
			\end{tikzpicture}
		\end{equation*}
		with $\s(g)=\tg(k_\s)$, $\tg(g)=\tg(k_\tg)$ and with bottom edge $\Phi(k_\tg\inv)\cdot g\cdot \Phi(k_\s)$.

The core of $(\Lambda,K,G,M)$ consists of elements $(k,1^G_{\s(k)},1^K_{\s(k)})$
\begin{equation*}
			\begin{tikzpicture}
				\coordinate (A) at (0,0) ;
				\coordinate (B) at (0,1) ;
				\coordinate (C) at (1,0) ;
				\coordinate (D) at (1,1) ;

				\draw[-] (A) to node[left ]{$k$} (B);
				\draw[-] (C) to node[right]{$1^K_m$} (D);
				\draw[-] (C) to node[below]{$1^G_m$} (A);
				\draw[-] (D) to node[above]{} (B);
\node at (0.5,0.5) {};
			\end{tikzpicture}
		\end{equation*}
with multiplication of $k,k'\in K$ with $\s(k)=\tg(k')$ given by filling the top right and bottom left
squares in the diagram below and multiplying all obtained squares
together      
      \begin{equation*}
	\begin{tikzpicture}[x=-2.6cm,y=-2.6cm]
			\arrowsquare[0,0,1,1,${(1^K_{\s(k)}, \Phi(k'), 1^K_{\s(k')})}$,tlbr]
			\arrowsquare[1,0,1,1,${(k, 1^G_{\s(k)}, 1^K_{\s(k)})}$,tbr]
			\arrowsquare[0,1,1,1,${(k', 1^G_{\s(k')}, 1^K_{\s(k')})}$,lbr]
			\arrowsquare[1,1,1,1,${(k', 1^G_{\s(k')}, k')}$,br]
	\end{tikzpicture}
\end{equation*}
hence yielding
    \[ (k,1^G_{\s(k)},1^K_{\s(k)})\cdot(k', 1^G_{\s(k')}, 1^K_{\s(k')})=(kk', 1^G_{\s(k')}, 1^K_{\s(k')}).
      \]
      As a consequence, the core (Lie) groupoid is isomorphic to
      $K\rr M$ as a (Lie) groupoid.  Note that $\Lambda\rr G$ is the
      action groupoid of the (Lie) groupoid action of
      $K\times K\rr M\times M$ on $(\tg,s)\colon G\to M\times M$,
      $(k_\tg,k_\s)\cdot g=\Phi(k_\tg)g\Phi(k_\s^{-1})$.

\subsection{Comma Lie $2$-groupoids}
Artin and Mazur's bar construction \cite{MeTa11} applied to the nerve of a comma double Lie groupoid yields a \textbf{comma Lie $2$-groupoid} $W_\bullet\Lambda$, which is defined as follows.

\begin{proposition}\label{prop:comma-groupoids-bar-construction-layer-structure}
	Let $\Lambda$ be the comma Lie groupoid of $G \rr M$, $K \rr M$ and $\Phi\colon K \to G$.	
	Then for $r>0$, the  $r$-th layer of $W_\bullet\Lambda$ is given by
	\begin{equation*}
		\left\{\left.
			\begin{pmatrix}
				g_1 & g_2 & \dots & g_{r-1} & g_r \\
				k_{1,1} & k_{1,2} & \dots & k_{1,r-1} & k_{1,r} \\
				& k_{2,2} & \dots & k_{2,r-1} & k_{2,r} \\
				& & \ddots & \vdots & \vdots \\
				& & & k_{r-1, r-1} & k_{r-1,r}\\
				& & & & k_{r,r}
			\end{pmatrix}\in G^{(r)} \times K^{(1)} \times K^{(2)}\dots \times K^{(r)}\,
		\right|
		\,\begin{aligned}
			\tg(k_{1,j}) &= \s(g_j) && \forall\, 1 \leq j \leq r
		\end{aligned}
		\right\}
	\end{equation*}
	as a subset of $G^{(r)} \times K^{(1)} \times K^{(2)}\dots \times K^{(r)}$.
	The $K$-factors of the cartesian product correspond to the columns of the matrix.
	Since $\s$ and $\tg$ are smooth submersions, $W_r\Lambda$ is an embedded submanifold of the product 
	$G^{(r)} \times K^{(1)} \times K^{(2)}\dots \times K^{(r)}$.
\end{proposition}

The proof of this proposition follows from the structure of $\Lambda$. If \begin{equation*}
		\begin{tikzpicture}[y=-1.3cm, x=1.3cm]
			\node at (0.5,-0.4) {$g_1$};
			\draw[->] (1,0) to (0,0);

			\arrowsquare[1,0,1,1,$\lambda_{1,1}$,tlbr]

			\arrowsquare[2,0,1,1,$\lambda_{1,2}$,tlr]
			\arrowsquare[2,1,1,1,$\lambda_{2,2}$,tlbr]

			\node at (3.5,1) {$\dots$};
			\arrowsquare[4,0,1,1,$\lambda_{1,r-1}$,tlr]
			\arrowsquare[4,1,1,1,$\lambda_{2,r-1}$,tlbr]

			\node at (3.5,2.5) {$\ddots$};
			\node at (4.5,2.5) {$\vdots$};
			\arrowsquare[4,3,1,1,$\lambda_{r-1,r-1}$,tlbr]
		
			\node[text width=1cm] at (5.6, 4.5) {$k_{r,r}$};
			\draw[->] (5,5) to (5,4);
		\end{tikzpicture}
	\end{equation*}
	is a simplex in $W_r\Lambda$, 
	where $\lambda_{1,j}=(k_{1,j}, g_{j+1}, k_{1,j+1})$ equals
	\begin{equation*}
	\begin{tikzpicture}
				\coordinate (A) at (0,0) ;
				\coordinate (B) at (0,1) ;
				\coordinate (C) at (1,0) ;
				\coordinate (D) at (1,1) ;

				\draw[-] (A) to node[left ]{$k_{1,j}$} (B);
				\draw[-] (C) to node[right]{$k_{1,j+1}$} (D);
				\draw[-] (C) to node[below]{} (A);
				\draw[-] (D) to node[above]{$g_{j+1}$} (B);
\node at (0.5,0.5) {};
			\end{tikzpicture}
	\end{equation*}
	for $j=1, \ldots, r-1$, then the squares in the lower lines are completely determined by their top edges and their left sides.
	Hence, for $i\geq 2$ and $i\leq j\leq r-1$, 
	$\lambda_{i,j}$ equals 
\begin{equation*}
	\begin{tikzpicture}
				\coordinate (A) at (0,0) ;
				\coordinate (B) at (0,1) ;
				\coordinate (C) at (1,0) ;
				\coordinate (D) at (1,1) ;

				\draw[-] (A) to node[left ]{$k_{i,j}$} (B);
				\draw[-] (C) to node[right]{$k_{i,j+1}$} (D);
				\draw[-] (C) to node[below]{} (A);
				\draw[-] (D) to node[above]{$g_{i,j}$} (B);
\node at (0.5,0.5) {};
			\end{tikzpicture}
	\end{equation*}
	with 
	\begin{equation*}
		g_{i,j}
		\coloneqq \left(
			\prod_{s=1}^{i-1} \Phi(k_{s,j})
		\right)^{-1} \cdot
		g_{j+1}
		\cdot \left(
			\prod_{s=1}^{i-1} \Phi(k_{s,j+1})
		\right),
	\end{equation*}
where the products are ordered from left to right by the indices.

\bigskip

Before moving on to the face maps of the bar manifold, consider a few layers explicitly:
The $0$-th layer is as always
	$W_0\Lambda
	= M$
and the first layer is 
	$W_1\Lambda
	= G \lrtimes{\s}{\tg} K$,
which corresponds to the set of partial squares
\begin{equation*}
	\begin{tikzpicture}[y=-1cm]
		\node at (0.5, -0.4){$g_1$};
		\draw[->] (1,0) to (0,0); 
		\draw[->] (1,1) to (1,0); 
		\node[text width=0.8cm] at (1.5, 0.5) {$k_{1,1}$};
	\end{tikzpicture}
\end{equation*}
which are written as $(2 \times 1)$-matrices
\begin{equation*}
	\left(
		\begin{matrix}
			g_1
			\\
			k_{1,1}
		\end{matrix}
	\right)\text{,}
\end{equation*}
with $\s(g_1) = \tg(k_{1,1}) \in M$.
The second layer consists of matrices
\begin{equation*}
	\begin{pmatrix}
			g_1 & g_2 \\
			k_{1,1} & k_{1,2} \\
			& k_{2,2}
	\end{pmatrix}
	\qquad
	\text{ with }
	\qquad
	\begin{aligned}
	&g_1,g_2\in G, k_{1,1}, k_{1,2}, k_{2,2}\in K\\
		&\s(g_1) = \tg(g_2)\text{, }
		 \s(k_{1,2}) = \tg(k_{2,2})\text{, }
		\\
		&\tg(k_{1,1}) = \s(g_1)\text{, }
		 \tg(k_{1,2}) = \s(g_2).
	\end{aligned}
\end{equation*}
A simplex in $W_3\Lambda$ is a matrix 
\begin{equation*}
	\left(
		\begin{matrix}
			g_1 & g_2 & g_3\\
			k_{1,1} & k_{1,2} & k_{1,3} \\
			& k_{2,2} & k_{2,3} \\
			& & k_{3,3}
		\end{matrix}
	\right)
	\quad
	\text{ with }
	\quad
	\begin{aligned}
	\s(g_1) &= \tg(g_2),
	& \s(g_2) &= \tg(g_3),
	& &
	\\
	\s(k_{1,2}) &= \tg(k_{2,2}),
	& \s(k_{1,3}) &= \tg(k_{2,3}),
	& \s(k_{2,3}) &= \tg(k_{3,3}),
	\\
	\tg(k_{1,1}) &= \s(g_1),
	& \tg(k_{1,2}) &= \s(g_2),
	& \tg(k_{1,3}) &= \s(g_3).
	\end{aligned}
\end{equation*}
\bigskip

For $r>1$ and $0 < i < r$, the $i$-th face map $f_i^r\colon W_r\Lambda\to W_{r-1}\Lambda$ is the map sending 
\begin{equation*}\label{eq:comma-groupoid-M-r}
	 \begin{pmatrix}
		g_1 & g_2 & \dots & g_{r-1} & g_r \\
		k_{1,1} & k_{1,2} & \dots & k_{ 1,r-1} & k_{1,r} \\
		 & k_{2,2} & \dots & k_{ 2,r-1} & k_{2,r} \\
		& & \ddots & \vdots & \vdots \\
		& & & k_{ r-1,r-1} & k_{r-1,r}\\
		& & & & k_{r,r}
	\end{pmatrix}
\end{equation*}
to
\begin{equation*}
	\begin{pmatrix}
		g_1 & g_2 & \dots & g_{i-1} & g_{i} \cdot g_{i+1} & g_{i+2} & \dots & g_{r-1} & g_r \\
		k_{1,1} & k_{1,2} & \dots & k_{1,i-1} & k_{1,i+1} & k_{1,i+2} & \dots & k_{1,r-1} & k_{1,r} \\
		& k_{2,2} & \dots & k_{2,i-1} & k_{2,i+1} & k_{2,i+2} & \dots & k_{2,r-1} & k_{2,r} \\
		& & \ddots & \vdots & \vdots & \vdots & & \vdots & \vdots \\
		& & & k_{i-1,i-1} & k_{i-1,i+1} & k_{i-1,i+2} & \dots & k_{i-1,r-1} & k_{i-1,r} \\
		& & &
		&k_{i,i+1} \cdot k_{i+1,i+1}
		& k_{i,i+2} \cdot k_{i+1,i+2}
		&\dots
		& k_{i,r-1} \cdot k_{i+1,r-1}
		& k_{i,r} \cdot k_{i+1,r} \\
		& & & & & k_{i+2,i+2} & \dots & k_{i+2,r-1} & k_{i+2,r} \\
		& & & & & & \ddots & \vdots & \vdots \\
		& & & & & & & k_{r-1,r-1} & k_{r-1,r} \\
		& & & & & & & & k_{r,r}
	\end{pmatrix}.
\end{equation*}
For $i=0$ and $i=r$ the maps $f_0^r, f_r^r\colon W_r\Lambda\to W_{r-1}\Lambda$ delete the first row and the last column of the structure respectively.
The same element is mapped here by $f_r^r$  to
\begin{equation*}
	\begin{pmatrix}
		g_1 & g_2 & \dots & g_{r-2} & g_{r-1} \\
		k_{1,1} & k_{1,2}& \dots & k_{1,r-2} & k_{1,r-1} \\
		& k_{2,2}& \dots & k_{2,r-2} & k_{2,r-1} \\
		& & \ddots & \vdots & \vdots \\
		& & & k_{r-2,r-2} & k_{r-2,r-1} \\
		& & & & k_{r-1,r-1} \\
	\end{pmatrix}
	\in W_{r-1}\Lambda
\end{equation*}
and by $f_0^r$ to 
\begin{equation*}
	\begin{pmatrix}
		\bar{g}_2 & \bar{g}_3 & \dots & \bar{g}_{r-1} & \bar{g}_r \\
		& k_{2,2} & \dots & k_{2,r-1} & k_{2,r} \\
		& & \ddots & \vdots & \vdots \\
		& & & k_{r-1,r-1} & k_{r-1,r}\\
		& & & & k_{r,r}
	\end{pmatrix}
	\in W_{r-1}\Lambda,
\end{equation*}
where $\bar{g}_j = \Phi(k_{1,j-1})^{-1} g_j \Phi(k_{1,j}) \in G$ for $1 < j \leq r$.

In particular
\begin{align*}
	f^1_0\colon &W_1\Lambda \to W_0\Lambda=M, \quad \begin{pmatrix} g_1 \\ k_{1,1} \end{pmatrix} \mapsto \s(k_{1,1}) \quad \text{ and}\\
	f^1_1\colon &W_1\Lambda \to W_0\Lambda=M, \quad \begin{pmatrix}
		g_1 \\ k_{1,1}
	\end{pmatrix} \mapsto \tg(g_1).
\end{align*}
For the next level, the images of the $2$-simplex
under the maps $f^2_0, f^2_1, f^2_2\colon W_2\Lambda \to W_1\Lambda$ are
\begin{equation*}
	\begin{tikzpicture}[y=-1cm]
	\begin{scope}[shift={(0,0)}]
		\draw[dotted] (1,0) to (0,0);
		\draw[dotted] (2,0) to (1,0);
		\draw[dotted] (1,1) to (1,0);
		\draw[dotted] (2,1) to (2,0);
		\node at (1.5,0.6) {$\Phi(k_{1,1})^{-1}\cdot g_2\cdot \Phi(k_{1,2})$};
		\draw[->] (2,1) to (1,1);
		\node at (2.4,1.5) {$k_{2,2}$};
		\draw[->] (2,2) to (2,1);
		\node at (1, 2.5) {$f_0^2$};
	\end{scope}
	\begin{scope}[shift={(5,0)}]
		\node at (1,-0.4) {$g_1 \cdot g_2$};
		\draw[->] (2,0) to (0,0);
		\draw[dotted] (1,1) to (1,0);
		\draw[dotted] (2,1) to (1,1);
		\node[text width=1.5cm] at (3,1) {$k_{1,2} \cdot k_{2,2}$};
		\draw[->] (2,2) to (2,0);
		\node at (1, 2.5) {$f_1^2$};
	\end{scope}
	\begin{scope}[shift={(10,0)}]
		\node at (0.5, -0.4) {$g_1$};
		\draw[->] (1,0) to (0,0);
		\draw[dotted] (2,0) to (1,0);
		\node at (1.4, 0.5) {$k_{1,1}$};
		\draw[->] (1,1) to (1,0);
		\draw[dotted] (2,1) to (2,0);
		\draw[dotted] (2,1) to (1,1);
		\draw[dotted] (2,2) to (2,1);
		\node at (1, 2.5) {$f_2^2$};
	\end{scope}
	\end{tikzpicture}
\end{equation*}
That is, 
\begin{align*}
	f^2_0\colon W_2\Lambda \to W_1\Lambda,\,
	&\begin{pmatrix}
		g_1 & g_2 \\
		k_{1,1} & k_{1,2} \\
		& k_{2,2}
	\end{pmatrix}
	\mapsto
	\begin{pmatrix}
		\scriptstyle \Phi(k_{1,1})^{-1} g_2 \Phi(k_{1,2}) \\
		k_{2,2}
	\end{pmatrix},
	\\
	f^2_1\colon W_2\Lambda \to W_1\Lambda,\,
	&\begin{pmatrix}
		g_1 & g_2 \\
		k_{1,1} & k_{1,2} \\
		& k_{2,2}
	\end{pmatrix}
	\mapsto
	\begin{pmatrix}
		g_1 \cdot g_2 \\
		k_{1,2} \cdot k_{2,2}
	\end{pmatrix},
	\\
	f^2_2\colon W_2\Lambda \to W_1\Lambda,\,
	&\begin{pmatrix}
		g_1 & g_2 \\
		k_{1,1} & k_{1,2} \\
		& k_{2,2}
	\end{pmatrix}
	\mapsto
	\begin{pmatrix}
		g_1\\
		k_{1,1}
	\end{pmatrix}
\end{align*}
and 
\begin{align*}
	f^3_0\colon W_3\Lambda \to W_2\Lambda,\quad 
	&\begin{pmatrix}
		g_1 & g_2 & g_3 \\
		k_{1,1} & k_{1,2} & k_{1,3} \\
		& k_{2,2} & k_{2,3} \\
		& & k_{3,3}
	\end{pmatrix}
	\mapsto
	\begin{pmatrix}
		\scriptstyle \Phi(k_{1,1})^{-1} g_2 \Phi(k_{1,2}) &
		\scriptstyle \Phi(k_{1,2})^{-1} g_3 \Phi(k_{1,3}) \\
		k_{2,2} & k_{2,3} \\
		& k_{3,3}
	\end{pmatrix}\text{ ,}
	\\
	f^3_1\colon W_3\Lambda \to W_2\Lambda,\quad 
	&\begin{pmatrix}
		g_1 & g_2 & g_3 \\
		k_{1,1} & k_{1,2} & k_{1,3} \\
		& k_{2,2} & k_{2,3} \\
		& & k_{3,3}
	\end{pmatrix}
	\mapsto
	\begin{pmatrix}
		g_1 \cdot g_2 & g_3 \\
		k_{1,2} \cdot k_{2,2} & k_{1,3} \cdot k_{2,3} \\
		& k_{3,3} \\
	\end{pmatrix},
	\\
	f^3_2\colon W_3\Lambda \to W_2\Lambda,\quad 
	&\begin{pmatrix}
		g_1 & g_2 & g_3 \\
		k_{1,1} & k_{1,2} & k_{1,3} \\
		& k_{2,2} & k_{2,3} \\
		& & k_{3,3}
	\end{pmatrix}
	\mapsto
	\begin{pmatrix}
		g_1 & g_2 \cdot g_3 \\
		k_{1,1} & k_{1,3} \\
		& k_{2,3} \cdot k_{3,3}
	\end{pmatrix},
	\\
	f^3_3\colon W_3\Lambda \to W_2\Lambda,\quad 
	&\begin{pmatrix}
		g_1 & g_2 & g_3 \\
		k_{1,1} & k_{1,2} & k_{1,3} \\
		& k_{2,2} & k_{2,3} \\
		& & k_{3,3}
	\end{pmatrix}
	\mapsto
	\begin{pmatrix}
		g_1 & g_2 \\
		k_{1,1} & k_{1,2} \\
		& k_{2,2}
	\end{pmatrix}.
\end{align*}
The reader is invited to check using these formulas that the horn maps $\lambda_{3,i}$ are diffeomorphisms for $i=0,1,2,3$, 
see also \cite[Theorem 4.5]{MeTa11}. Similarly, higher faces are not particularly relevant either as they yield diffeomorphic horn maps by \cite[Theorem 4.5]{MeTa11}.
For $i=0,1,2$, the $(2,i)$-horn maps are given by 
\begin{align*}
	\lambda_{2,0}&=(f^2_1, f^2_2)\colon \begin{pmatrix}
		g_1 & g_2 \\
		k_{1,1} & k_{1,2} \\
		& k_{2,2}
	\end{pmatrix} \mapsto
	\left(
		\begin{pmatrix}
			g_1 \cdot g_2 \\
			k_{1,2} \cdot k_{2,2}
		\end{pmatrix},
		\begin{pmatrix}
			g_1 \\
			k_{1,1}
		\end{pmatrix}
	\right),\\
	\lambda_{2,1}&=(f^2_0, f^2_2)\colon \begin{pmatrix}
		g_1 & g_2 \\
		k_{1,1} & k_{1,2} \\
		& k_{2,2}
	\end{pmatrix} \mapsto
	\left(
		\begin{pmatrix}
			\Phi(k_{1,1})\inv\cdot g_2\cdot \Phi(k_{1,2}) \\
			 k_{2,2}
		\end{pmatrix},
		\begin{pmatrix}
			g_1 \\
			k_{1,1}
		\end{pmatrix}
	\right), \\
\text{ and }\quad  	\lambda_{2,2}&=(f^2_0, f^2_1)\colon \begin{pmatrix}
		g_1 & g_2 \\
		k_{1,1} & k_{1,2} \\
		& k_{2,2}
	\end{pmatrix} \mapsto
	\left(
		\begin{pmatrix}
			\Phi(k_{1,1})\inv\cdot g_2\cdot \Phi(k_{1,2}) \\
			 k_{2,2}
		\end{pmatrix},
		\begin{pmatrix}
			g_1\cdot g_2 \\
			k_{1,2}\cdot k_{2,2}
		\end{pmatrix}
	\right).
\end{align*}
These maps are not diffeomorphisms, but it is easy to see that they are surjective submersions.

\bigskip

\bigskip
The degeneracy maps of $W_\bullet \Lambda$ are given as follows.
For $0 < i < r$, the $i$-th degeneracy map of the $r$-th layer, $\delta_i^r\colon W_r\Lambda \to W_{r+1}\Lambda$ maps  a simplex as described and denoted in Proposition \ref{prop:comma-groupoids-bar-construction-layer-structure} to 
\begin{equation*}
	\begin{pmatrix}
		g_1 & g_2 & \dots & g_{i} & 1^G_{\s(g_i)} & g_{i+1} & \dots & g_{r-1} & g_r \\
		k_{1,1} & k_{1,2} & \dots & k_{1,i} & k_{1,i} & k_{1,i+1} & \dots & k_{1,r-1} & k_{1,r} \\
		 & k_{2,2} & \dots & k_{2,i} &k_{2,i} & k_{2,i+1} & \dots & k_{2,r-1} & k_{2,r} \\
		& & \ddots & \vdots & \vdots & \vdots & & \vdots & \vdots \\
		& & & k_{i,i} & k_{i,i} & k_{i,i+1} & \dots & k_{i,r-1}& k_{i,r}\\
		& & & & 1^K_{\s(k_{i,i})} & 1^K_{\s(k_{i,i+1})} & \dots &1^K_{\s(k_{i,r-1})} & 1^K_{\s(k_{i,r})} \\
		& & & & & k_{i+1,i+1} & \dots & k_{i+1,r-1} & k_{i+1,r}\\
		& & & & & & \ddots & \vdots & \vdots \\
		& & & & & & & k_{r-1,r-1} & k_{r-1,r} \\
		& & & & & & & & k_{r,r}
	\end{pmatrix}.
\end{equation*}

For $i=0$, the same element of $W_r\Lambda$ is sent by $\delta_0^r$ to
\begin{equation*}
	 \begin{pmatrix}
		{1^G_{\tg(g_1)}} & g_1 & g_2 & \dots & g_{r-1} & g_r \\
		{1^K_{\tg(g_1)}} & {1^K_{\tg(g_2)}} & {1^K_{\tg(g_3)}} & \dots & {1^K_{\tg(g_r)}} & {1^K_{\s(g_r)}}\\
		& k_{1,1} & k_{1,2} & \dots & k_{ 1,r-1} & k_{1,r} \\
		&  & k_{2,2} & \dots & k_{ 2,r-1} & k_{2,r} \\
		& & & \ddots & \vdots & \vdots \\
		& & & & k_{ r-1,r-1} & k_{r-1,r}\\
		& & & & & k_{r,r}
	\end{pmatrix} \in W_{r+1}\Lambda
\end{equation*}
and by $\delta_r^r$ to 
\begin{equation*}
	 \begin{pmatrix}
		g_1 & g_2 & \dots & g_{r-1} & g_r &{ 1^G_{\s(g_r)} }\\
		k_{1,1} & k_{1,2} & \dots & k_{ 1,r-1} & k_{1,r} &{k_{1,r}}\\
		& k_{2,2} & \dots & k_{ 2,r-1} & k_{2,r} &{ k_{2,r}}\\
		& & \ddots & \vdots & \vdots &{\vdots}\\
		& & & k_{ r-1,r-1} & k_{r-1,r} &{ k_{r-1,r}}\\
		& & & & k_{r,r} &{ k_{r,r} }\\
		& & & & &{ 1^K_{\s(k_{r,r})} }\\
	\end{pmatrix} \in W_{r+1}\Lambda.
\end{equation*}

\section{The bar construction applied to transitive double Lie groupoids}
This section briefly recalls the construction of transitive double Lie groupoids from their core diagrams and describes in detail the \emph{transitive Lie $2$-groupoids} produced by applying the bar construction to this class of examples.
\subsection{Transitive double Lie groupoids}
Consider a transitive core diagram as in \eqref{trans_core_d}.  Construct the comma double Lie groupoid $\Lambda:=(\da_G,\da_G)$
    \[\begin{tikzcd}
	\Lambda & G \\
	K & M
	\arrow[shift left=1, from=1-2, to=2-2]
	\arrow[shift right=1, from=1-2, to=2-2]
	\arrow[shift right=1, from=2-1, to=2-2]
	\arrow[shift left=1, from=2-1, to=2-2]
	\arrow[shift right=1, from=1-1, to=1-2]
	\arrow[shift right=1, from=1-1, to=2-1]
	\arrow[shift left=1, from=1-1, to=2-1]
	\arrow[shift left=1, from=1-1, to=1-2]
      \end{tikzcd}\]
    with core $K$.
    Next, $G\to M$ acts on $K^G:=\ker(\da_H)\subseteq K$ via
    \[ \rho\colon G\lrtimes{\s}{\tg}  K^G\to K^G, \qquad \rho(g)(\kappa)=k\cdot \kappa \cdot k\inv
    \]
    for any $k\in K$ such that $\da_G(k)=g$.
    Consider the closed, embedded, wide and normal Lie subgroupoid
    \begin{equation*}\label{def_N}
      N:=\{(\kappa_2,g,\kappa_1)\in K^G\times G\times K^G\mid \rho(g)(\kappa_1)=\kappa_2\}
    \end{equation*}
    of $\Lambda\rr G$. The quotient $\Theta=\Lambda/N$ then has a Lie
    groupoid structure over $G$. The elements of $\Theta$ are classes
    \[ \langle k_2, g, k_1\rangle=\{(k_2\kappa_2, g, k_1\kappa_1)\mid (\kappa_2,g,\kappa_1)\in N\}
    \]
    with $(k_2,g,k_1)\in\Lambda$.
    Set $\s_H,\tg_H\colon \Lambda/N\to H$ 
    \[ \s_H\langle k_2, g, k_1\rangle=\da_H(k_1), \qquad \tg_H\langle k_2, g, k_1\rangle=\da_H(k_2),
    \]
    and a partial multiplication $\cdot_H\colon \Theta\lrtimes{\s_H}{\tg_H} \Theta\to \Theta$ by 
    \[ \langle k_2', g', k_1'\rangle\cdot_H\langle k_2, g,
      k_1\rangle=\langle k_2', g', k_1'\rangle\cdot_H\langle k_1',
      g, k_1\rho(g\inv)(k_2\inv k_1')\rangle=\langle k_2', g'g, k_1\rho(g\inv)(k_2\inv k_1')\rangle.
    \]
    It is easy to
    check that $\Theta=\Lambda/N\rr H$ becomes a groupoid with these structure maps and with the inversion
    \[ \langle k_2, g, k_1\rangle\mapsto \langle k_1, g\inv,k_2\rangle
    \]
    and the unit inclusion
    \[ h\mapsto \langle k, 1_{\s(h)},k\rangle
    \]
    for any $k\in K$ such that $\da_H(k)=h$.
    Since those structure maps are defined such that
    \[\begin{tikzcd}
	\Lambda & {\Theta=\Lambda/N} \\
	K & H
	\arrow[shift left=1, from=1-2, to=2-2]
	\arrow[shift right=1, from=1-2, to=2-2]
	\arrow["{\partial_H}"', from=2-1, to=2-2]
	\arrow["\pi", from=1-1, to=1-2]
	\arrow[shift right=1, from=1-1, to=2-1]
	\arrow[shift left=1, from=1-1, to=2-1]
      \end{tikzcd}\] is a morphism of groupoids and the projections
    $\pi\colon \Lambda\to\Theta$ and $\da_H\colon K\to H$ are
    surjective submersions, it is easy to check that $\Theta\rr H$ is
    a Lie groupoid. The interchange law and the surjectivity of the
    double source map are also easily deduced from the one in
    $\Lambda=(\da_G,\da_G)$, and $(\Theta, G,H,M)$ is a double Lie
    groupoid. Its core consists of classes
    $\langle k, 1_{\s(k)},\kappa\rangle=\langle k\kappa\inv,
    1_{\s(k)},1_{\s(k)}\rangle$, and is obviously isomorphic to $K$ as
    a Lie groupoid. The core diagram of $(\Theta, G,H,M)$ is hence
    again \[\begin{tikzcd}
	K & G \\
	H & M \arrow[shift left=1, from=1-2, to=2-2] \arrow[shift
        right=1, from=1-2, to=2-2] \arrow[shift right=1, from=2-1,
        to=2-2] \arrow[shift left=1, from=2-1, to=2-2] \arrow[shift
        right=1, from=1-1, to=2-2] \arrow[shift left=1, from=1-1,
        to=2-2] \arrow["{\partial_G}", from=1-1, to=1-2]
        \arrow["{\partial_H}"', from=1-1, to=2-1]
      \end{tikzcd}\]
    by construction.
 Conversely, given a transitive double Lie groupoid
     \[\begin{tikzcd}
	S& G \\
	H & M
	\arrow[shift left=1, from=1-2, to=2-2]
	\arrow[shift right=1, from=1-2, to=2-2]
	\arrow[shift right=1, from=2-1, to=2-2]
	\arrow[shift left=1, from=2-1, to=2-2]
	\arrow[shift right=1, from=1-1, to=1-2]
	\arrow[shift right=1, from=1-1, to=2-1]
	\arrow[shift left=1, from=1-1, to=2-1]
	\arrow[shift left=1, from=1-1, to=1-2]
      \end{tikzcd}\] with core $K$ and core diagram as in \eqref{trans_core_d}, set $\Psi_S\colon \Theta\to S$,
    $\langle k_2, g, k_1\rangle\mapsto k_2\cdot_H 1^S_g\cdot_Hk_1^{-1}\in S$. Then $\Psi_S$ is an isomorphism of double Lie groupoids,
    see \cite{BrMa92,JoMa21}.

    A morphism $\Phi\colon \Gamma_1\to \Gamma_2$ of double Lie
    groupoids with side morphisms $\phi_G\colon G_1\to G_2$,
    $\phi_H\colon H_1\to H_2$ and with core morphism
    $\phi_K\colon K_1\to K_2$, all three above the identity on $M$, induces a morphism of the
    corresponding core diagrams as in the following diagram.
    \begin{equation}\label{cor_mor_gpd}\begin{tikzcd}
	{K_1} &&&& {K_2} \\
	&& {G_1} &&&& {G_2} \\
	{H_1} &&&& {H_2} \\
	&& M &&&& M
	\arrow["{\phi_K}", from=1-1, to=1-5]
	\arrow["{\phi_H}"{pos=0.7}, from=3-1, to=3-5]
	\arrow["{\partial_{H_1}}"{description}, from=1-1, to=3-1]
	\arrow["{\partial_{H_2}}"{description, pos=0.7}, from=1-5, to=3-5]
	\arrow["{\partial_{G_1}}"{description}, from=1-1, to=2-3]
	\arrow["{\partial_{G_2}}"{description}, from=1-5, to=2-7]
	\arrow[shift left=1, from=3-1, to=4-3]
	\arrow[shift right=1, from=3-1, to=4-3]
	\arrow[shift right=1, from=3-5, to=4-7]
	\arrow[shift left=1, from=3-5, to=4-7]
	\arrow[shift right=1, from=2-7, to=4-7]
	\arrow[shift left=1, from=2-7, to=4-7]
	\arrow[shift right=1, from=2-3, to=4-3]
	\arrow[shift left=1, from=2-3, to=4-3]
	\arrow[shift right=1, from=1-1, to=4-3]
	\arrow[shift left=1, from=1-1, to=4-3]
	\arrow[shift right=1, from=1-5, to=4-7]
	\arrow[shift left=1, from=1-5, to=4-7]
	\arrow["{\operatorname{id}_M}"{description}, from=4-3, to=4-7]
	\arrow["{\phi_G}"{pos=0.3}, from=2-3, to=2-7]
      \end{tikzcd}\end{equation}  Consider conversely a morphism of
    transitive core diagrams (of Lie groupoids) as in
    \eqref{cor_mor_gpd}.  Let $\Lambda_i$ be the total space of the
    comma double Lie groupoid $(\da_{G_i},\da_{G_i})$ for $i=1,2$ and
    set $\Phi\colon \Lambda_1\to \Lambda_2$,
    $\Phi(k_1,g,k_2)=(\phi_K(k_1),\phi_G(g),\phi_K(k_2))$. The map
    $\Phi$ is obviously a morphism of double Lie groupoids, and a
    computation shows $\Phi(N_1)=N_2$ for the normal subgroupoids
    $N_i$ of $\Lambda_i\rr G_i$  for
    $i=1,2$.  Therefore, it induces a morphism of the transitive
    double Lie groupoids $\overline{\Phi}\colon \Theta_1=\Lambda_1/N_1\to \Lambda_2/N_2=\Theta_2$ with core
    morphism given by \eqref{cor_mor_gpd}. 
    
    This establishes an equivalence between the category of transitive core diagrams and the category of transitive double Lie groupoids, see \cite{BrMa92,JoMa21}.
    In particular, if \[\begin{tikzcd}
	K & G \\
	H & M \arrow[shift left=1, from=1-2, to=2-2] \arrow[shift
        right=1, from=1-2, to=2-2] \arrow[shift right=1, from=2-1,
        to=2-2] \arrow[shift left=1, from=2-1, to=2-2] \arrow[shift
        right=1, from=1-1, to=2-2] \arrow[shift left=1, from=1-1,
        to=2-2] \arrow["{\partial_G}", from=1-1, to=1-2]
        \arrow["{\partial_H}"', from=1-1, to=2-1]
      \end{tikzcd}\]
is the core diagram of a transitive double Lie groupoid $(\Theta, G, H, M)$, then 
\[\begin{tikzcd}
	K & G \\
	K & M \arrow[shift left=1, from=1-2, to=2-2] \arrow[shift
        right=1, from=1-2, to=2-2] \arrow[shift right=1, from=2-1,
        to=2-2] \arrow[shift left=1, from=2-1, to=2-2] \arrow[shift
        right=1, from=1-1, to=2-2] \arrow[shift left=1, from=1-1,
        to=2-2] \arrow["{\partial_G}", from=1-1, to=1-2]
        \arrow["{\id_K}"', from=1-1, to=2-1]
      \end{tikzcd}\]
      is a transitive core diagram as well, which corresponds to the (transitive) comma double Lie groupoid $\Lambda=(\partial_G,\partial_G)$, and 
        \begin{equation*}
        \begin{tikzcd}
	{K} &&&& {K} \\
	&& {G} &&&& {G} \\
	{K} &&&& {H} \\
	&& M &&&& M
	\arrow["{\id_K}", from=1-1, to=1-5]
	\arrow["{\partial_H}"{pos=0.7}, from=3-1, to=3-5]
	\arrow["{\id_K}"{description}, from=1-1, to=3-1]
	\arrow["{\partial_{H}}"{description, pos=0.7}, from=1-5, to=3-5]
	\arrow["{\partial_{G}}"{description}, from=1-1, to=2-3]
	\arrow["{\partial_{G}}"{description}, from=1-5, to=2-7]
	\arrow[shift left=1, from=3-1, to=4-3]
	\arrow[shift right=1, from=3-1, to=4-3]
	\arrow[shift right=1, from=3-5, to=4-7]
	\arrow[shift left=1, from=3-5, to=4-7]
	\arrow[shift right=1, from=2-7, to=4-7]
	\arrow[shift left=1, from=2-7, to=4-7]
	\arrow[shift right=1, from=2-3, to=4-3]
	\arrow[shift left=1, from=2-3, to=4-3]
	\arrow[shift right=1, from=1-1, to=4-3]
	\arrow[shift left=1, from=1-1, to=4-3]
	\arrow[shift right=1, from=1-5, to=4-7]
	\arrow[shift left=1, from=1-5, to=4-7]
	\arrow["{\operatorname{id}_M}"{description}, from=4-3, to=4-7]
	\arrow["{\id_G}"{pos=0.3}, from=2-3, to=2-7]
      \end{tikzcd}\end{equation*}
      is then a morphism of transitive core diagrams. The corresponding morphism of (transitive) double Lie groupoids is the surmersive smooth map $\pi\colon \Lambda\to\Theta$, $(k_1,g,k_2)\mapsto \langle k_1, g, k_2\rangle$.

	\subsection{Transitive Lie $2$-groupoids}
	
	Let $(\Theta, G, H, M)$ be a transitive double Lie groupoid with core $K$, core diagram as in \eqref{trans_core_d} and set $\Lambda:=(\partial_G,\partial_G)$. By the previous section, 
	$\Theta$ is isomorphic as a double Lie groupoid to $\Lambda/N$, and the projection $\pi\colon\Lambda\to \Theta=\Lambda/N$ is a morphism of double Lie groupoids. Denote by $W_\bullet\Theta$ the simplicial manifold obtained by applying Artin-Mazur's bar construction to $(\Theta, G, H, M)$ as in \cite{MeTa11}, and by $\overline f^r_i$ and $\overline \delta^r_i$ its faces and degeneracy maps. The faces and degeneracy maps of $W_\bullet\Lambda$ are denoted by $f^r_i$, $\delta^r_i$ as before.
The normal subgroupoid $N\rr G$ of $\Lambda \rr G$ defines as follows an equivalence relation on $W_\bullet\Lambda$.
	For $r\geq 1$ an element   \begin{equation*}
		x:=\begin{pmatrix}
			g_1 & g_2 & \dots & g_{r-1} & g_r\\
			k_{1,1} & k_{1,2} & \dots & k_{1,r-1} & k_{1,r}\\
			& k_{2,2} & \dots & k_{2,r-1} & k_{2,r}\\
			& & \ddots & \vdots & \vdots \\
			& & & k_{r-1,r-1} & k_{r-1,r} \\
			& & & & k_{r,r} \\
		\end{pmatrix}\in W_r\Lambda
	\end{equation*}
	is equivalent to 
	\[\begin{pmatrix}
			g_1 & g_2 & \dots & g_{r-1} & g_r\\
			k_{1,1}\kappa_{1,1}  & k_{1,2}  \kappa_{1,2}& \dots &  k_{1,r-1} \kappa_{1,r-1}&  k_{1,r}\kappa_{1,r}\\
			& k_{2,2} \kappa_{2,2} & \dots &  k_{2,r-1} \kappa_{2,r-1}& k_{2,r}\kappa_{2,r} \\
			& & \ddots & \vdots & \vdots \\
			& & &  k_{r-1,r-1} \kappa_{r-1,r-1}&  k_{r-1,r} \kappa_{r-1,r}\\
			& & & & k_{r,r} \kappa_{r,r} \\
		\end{pmatrix}
		\]
		for all $\kappa_{i,j} \in K^G$, $1 \leq i \leq j \leq r$, such that $\rho(g_{i,j+1})(\kappa_{i,j+1}) = \kappa_{i,j}$, where as before
		\begin{equation*}
		g_{i,j+1} =
		\left(\prod_{l=1}^i \partial_G(k_{l, j})\right)^{-1}
		\cdot
		g_{j+1} 
		\cdot
		\left(\prod_{l=1}^i \partial_G(k_{l, j+1})\right)
	\end{equation*}
	for $1\leq i\leq j<r$. 
	 Denote then by 
	\begin{equation*}
[x]:=\left\{\left.
		\begin{pmatrix}
			g_1 & g_2 & \dots & g_{r-1} & g_r\\
			 k_{1,1} \kappa_{1,1}&  k_{1,2} \kappa_{1,2}& \dots &  k_{1,r-1} \kappa_{1,r-1}& k_{1,r}\kappa_{1,r} \\
			& k_{2,2} \kappa_{2,2} & \dots & k_{2,r-1}  \kappa_{2,r-1}&  k_{2,r}\kappa_{2,r}\\
			& & \ddots & \vdots & \vdots \\
			& & & k_{r-1,r-1} \kappa_{r-1,r-1} &  k_{r-1,r}\kappa_{r-1,r} \\
			& & & &  k_{r,r}\kappa_{r,r} \\
		\end{pmatrix}
		\right|
		\begin{gathered}
		\kappa_{i,j} \in K^G\\
		\forall\, 1 \leq i \leq j \leq r\\
		\\
		\rho(g_{i,j+1})(\kappa_{i,j+1}) = \kappa_{i,j}\\
		\forall\, 1 \leq i \leq j < r		
		\end{gathered}
		\right\}
	\end{equation*}
	the equivalence class of $x$.
	The following proposition is immediate and follows from the fact that the (surmersive) double Lie groupoid morphism $\pi\colon \Lambda\to \Theta$ with sides $\partial_H$ and $\id_G$ and core $\id_K$ defines a (layerwise surmersive) morphism $\pi_\bullet\colon W_\bullet\Lambda\to W_\bullet\Theta$ of simplicial manifolds.

	\begin{proposition}
		In the situation above, $W_0\Theta=M$ and for $r > 0$ the  manifold $W_r \Theta$ is defined by 
	\begin{equation*}
		\left\{\left.
			[x]
		\,\right|\,
			x\in W_r\Lambda		\right\}
	\end{equation*}
	with the unique smooth structure such that $\pi_r\colon W_r\Lambda\to W_r\Theta$, $x\mapsto [x]$, is a smooth surjective submersion.
	The faces and degeneracy maps of $W_\bullet\Theta$ are the unique maps $\overline f^r_i$, $\overline \delta^r_i$ such that the diagrams	
	\[\begin{tikzcd}
	{W_{r+1}\Lambda} && {W_{r+1}\Theta} && {W_r\Lambda} && {W_r\Theta} \\
	{W_r\Lambda} && {W_r\Theta} && {W_{r+1}\Lambda} && {W_{r+1}\Theta}
	\arrow["{\pi_{r+1}}", from=1-1, to=1-3]
	\arrow["{f^{r+1}_i}"', from=1-1, to=2-1]
	\arrow["{\overline f^{r+1}_i}", from=1-3, to=2-3]
	\arrow["{\pi_r}", from=1-5, to=1-7]
	\arrow["{\delta^r_i}"', from=1-5, to=2-5]
	\arrow["{\overline\delta^r_i}", from=1-7, to=2-7]
	\arrow["{\pi_r}"', from=2-1, to=2-3]
	\arrow["{\pi_{r+1}}"', from=2-5, to=2-7]
\end{tikzcd}\]
all commute for $r\geq 0$ and $i$ where defined.
	\end{proposition}

	Explicitly, the $0$-th layer is
\begin{equation*}
	W_0\Theta
	= M
\end{equation*}
and the first layer is 
\begin{equation*}
	W_1\Theta
	= G \lrtimes{\s}{\tg} H, 
\end{equation*}
via the smooth identification 
\begin{equation*}
		\begin{bmatrix}
			g_1
			\\
			k_{1,1}
		\end{bmatrix}=\begin{pmatrix}
		g_1\\
		\partial_H(k_{1,1})
		\end{pmatrix}
\end{equation*}
where $g_1\in G$ and $k_{1,1} \in K$ are such that $\s(g_1) = \tg(k_{1,1}) \in M$.
The second layer consists of classes
\begin{equation*}
	\begin{bmatrix}
			g_1 & g_2 \\
			k_{1,1} & k_{1,2} \\
			& k_{2,2}
	\end{bmatrix}
	\qquad
	\text{ with }
	\qquad
	\begin{aligned}
	&g_1,g_2\in G, k_{1,1}, k_{1,2}, k_{2,2}\in K\\
		&\s(g_1) = \tg(g_2)\text{, }
		 \s(k_{1,2}) = \tg(k_{2,2})\text{, }
		\\
		&\tg(k_{1,1}) = \s(g_1)\text{, }
		 \tg(k_{1,2}) = \s(g_2).
	\end{aligned}
\end{equation*}
\bigskip

The general face and degeneracy maps of $W_\bullet \Theta$  are the quotients of the corresponding maps $W_\bullet \Lambda$; the formulas can be obtained  by applying $\pi_r$ and replacing $k_{i,j}$ by their $\partial_H$-images where appropriate.
The following focuses instead on the lower-level maps.

The lower face maps are 
\begin{align*}
	\overline f^1_0\colon &W_1\Theta \to W_0\Theta=M, \quad \begin{pmatrix} g_1 \\ h_{1,1} \end{pmatrix} \mapsto \s(h_{1,1}) \quad \text{ and}\\
	\overline f^1_1\colon &W_1\Theta \to W_0\Theta=M, \quad \begin{pmatrix}
		g_1 \\ h_{1,1}
	\end{pmatrix} \mapsto \tg(g_1),
\end{align*}
and 
\begin{align*}
	\overline f^2_0\colon W_2\Theta \to W_1\Theta,\quad 
	&\begin{bmatrix}
		g_1 & g_2 \\
		k_{1,1} & k_{1,2} \\
		& k_{2,2}
	\end{bmatrix}
	\mapsto
	\begin{pmatrix}
		\scriptstyle \partial_G(k_{1,1})^{-1} g_2 \partial_G(k_{1,2}) \\
		\partial_H(k_{2,2})
	\end{pmatrix},
	\\
	\overline f^2_1\colon W_2\Theta \to W_1\Theta,\quad 
	&\begin{bmatrix}
		g_1 & g_2 \\
		k_{1,1} & k_{1,2} \\
		& k_{2,2}
	\end{bmatrix}
	\mapsto
	\begin{pmatrix}
		g_1 \cdot g_2 \\
		\partial_H(k_{1,2} \cdot k_{2,2})
	\end{pmatrix},
	\\
	\overline f^2_2\colon W_2\Theta \to W_1\Theta,\quad 
	&\begin{bmatrix}
		g_1 & g_2 \\
		k_{1,1} & k_{1,2} \\
		& k_{2,2}
	\end{bmatrix}
	\mapsto
	\begin{pmatrix}
		g_1\\
		\partial_H(k_{1,1})
	\end{pmatrix}.
\end{align*}
For $i=0,1,2$, the surmersive $(2,i)$-horn maps are hence given by 
\begin{align*}
	\overline\lambda_{2,0}&=\left(\overline f^2_1, \overline f^2_2\right)\colon \begin{bmatrix}
		g_1 & g_2 \\
		k_{1,1} & k_{1,2} \\
		& k_{2,2}
	\end{bmatrix} \mapsto
	\left(
		\begin{pmatrix}
			g_1 \cdot g_2 \\
			\partial_H(k_{1,2})\cdot \partial_H(k_{2,2})
		\end{pmatrix},
		\begin{pmatrix}
			g_1 \\
			\partial_H(k_{1,1})
		\end{pmatrix}
	\right),\\
	\overline\lambda_{2,1}&=\left(\overline f^2_0, \overline f^2_2\right)\colon \begin{bmatrix}
		g_1 & g_2 \\
		k_{1,1} & k_{1,2} \\
		& k_{2,2}
	\end{bmatrix} \mapsto
	\left(
		\begin{pmatrix}
			\partial_G(k_{1,1})\inv\cdot g_2\cdot \partial_G(k_{1,2}) \\
			\partial_H( k_{2,2})
		\end{pmatrix},
		\begin{pmatrix}
			g_1 \\
			\partial_H(k_{1,1})
		\end{pmatrix}
	\right), \\
\text{ and }\quad  	\overline\lambda_{2,2}&=\left(\overline f^2_0, \overline f^2_1\right)\colon \begin{bmatrix}
		g_1 & g_2 \\
		k_{1,1} & k_{1,2} \\
		& k_{2,2}
	\end{bmatrix} \mapsto
	\left(
		\begin{pmatrix}
			\partial_G(k_{1,1})\inv\cdot g_2\cdot \partial_G(k_{1,2}) \\
			 \partial_H(k_{2,2})
		\end{pmatrix},
		\begin{pmatrix}
			g_1\cdot g_2 \\
			\partial_H(k_{1,2})\cdot \partial_H(k_{2,2})
		\end{pmatrix}
	\right).
\end{align*}

\bigskip

In low levels, the degeneracy maps are given by 
\begin{equation*}
	\overline \delta^0_0\colon M \to W_1 \Theta, \quad x \mapsto \begin{pmatrix}
		1^G_x \\ 1^H_x
	\end{pmatrix},
\end{equation*}
\begin{align*}
	\overline\delta^1_0\colon &W_1 \Theta \to W_2 \Theta, \quad 
	\begin{pmatrix}
		g_1 \\ h_{1,1}
	\end{pmatrix}
	\mapsto
	\begin{bmatrix}
		1^G_{\tg(g_1)} & g_1 \\ 		
		1^K_{\tg(g_1)} & 1^K_{\s(g_1)} \\
		& k_{1,1}
	\end{bmatrix}
	\\
	\overline\delta^1_1\colon &W_1 \Theta \to W_2 \Theta, \quad 
	\begin{pmatrix}
		g_1 \\ h_{1,1}
	\end{pmatrix}
	\mapsto
	\begin{bmatrix}
		g_1 & 1^G_{\s(g_1)} \\
		k_{1,1} & k_{1,1} \\
		& 1^K_{\s(k_{1,1})}
	\end{bmatrix}
\end{align*}
for any choice of $k_{1,1}\in K$ such that $\partial_H(k_{1,1})=h_{1,1}$, 
and 
\begin{align*}
	\overline\delta^2_0\colon &W_2 \Theta \to W_3 \Theta, \quad 
	\begin{bmatrix}
		g_1 & g_2 \\
		k_{1,1} &k_{1,2} \\
		& k_{2,2}
	\end{bmatrix}
	\mapsto
	\begin{bmatrix}
		1^G_{\tg(g_1)} & g_1 & g_2\\ 		
		1^K_{\tg(g_1)} & 1^K_{\s(g_1)} & 1^K_{\s(g_2)} \\
		& k_{1,1} & k_{1,2} \\
		& & k_{2,2}
	\end{bmatrix}
	\\
	\overline\delta^2_1\colon &W_2 \Theta \to W_3\Theta, \quad 
	\begin{bmatrix}
		g_1 & g_2 \\
		k_{1,1} & k_{1,2} \\
		& k_{2,2}
	\end{bmatrix}
	\mapsto
	\begin{bmatrix}
		g_1 & 1^G_{\s(g_1)} & g_2 \\
		k_{1,1} & k_{1,1} & k_{1,2} \\
		& 1^K_{\s(k_{1,1})} & 1^K_{\s(k_{1,2})} \\
		& & k_{2,2}
	\end{bmatrix}
	\\
	\overline\delta^2_2\colon &W_2 \Theta \to W_3 \Theta, \quad 
	\begin{bmatrix}
		g_1 & g_2 \\
		k_{1,1} & k_{1,2} \\
		& k_{2,2}
	\end{bmatrix}
	\mapsto
	\begin{bmatrix}
		g_1 & g_2 & 1^G_{\s(g_2)} \\
		k_{1,1} & k_{1,2} & k_{1,2} \\
		& k_{2,2} & k_{2,2} \\
		& & 1^K_{\s(k_{2,2})}
	\end{bmatrix}.
\end{align*}

	\section{The bar construction applied to VB-groupoids}
	This section focuses on the class of examples given by (split) VB-groupoids, which arise as special examples of double Lie groupoids with an abelian side.
	\subsection{VB-groupoids and $2$-representations}
	This section considers VB-groupoids \cite{Pradines88,Mackenzie05}, i.e.~double Lie groupoids
	 \[\begin{tikzcd}
	{\Gamma} & G \\
	{E} & M
	\arrow[from=1-1, to=1-2]
	\arrow[shift left=1, from=1-1, to=2-1]
	\arrow[shift right=1, from=1-1, to=2-1]
	\arrow[shift right=1, from=1-2, to=2-2]
	\arrow[shift left=1, from=1-2, to=2-2]
	\arrow[from=2-1, to=2-2]
      \end{tikzcd}\]
where $q_\Gamma\colon \Gamma\to G$ and $q\colon E\to M$ are smooth vector bundles and all structure maps are compatible.
       The source and target maps $\Gamma\rr
        E$ are called $\tilde\s$ and $\tilde\tg$, respectively. The multiplication is $\tilde \m\colon \Gamma^{(2)}\to \Gamma$, the inversion $\tilde\iv\colon \Gamma\to\Gamma$ and the unit inclusion $\tilde\epsilon \colon E\to \Gamma$. 
       
Since $(\Gamma\rr E, G\rr M)$ is a VB-groupoid, the pairs $(\tilde \s, \s)$, $(\tilde\tg, \tg)$, $(\tilde \m, \m)$, $(\tilde \iv,\iv)$ and $(\tilde\epsilon, \epsilon)$ are all vector bundle morphisms.
The pair $(\tilde \s, \s)$ is in particular a vector bundle fibration, and so its kernel
\[\Gamma^{\s}:=\ker\tilde\s\subseteq \Gamma
\]
is a vector bundle over $G$ and its restriction
\[ C:=\Gamma^{\s}\an{M} 
\]
to $M$ is a vector bundle over $M$ with projection denoted by $q_C$.
Consider an element $\gamma_g\in \Gamma^{\s}_g$ for some $g\in G$. Then $\tilde\s(\gamma_g)=0_{\s(g)}$, which equals $\tilde\tg \left(0_{g\inv}\right)$.
The product $\gamma_g\cdot0_{g\inv}=\tilde \m(\gamma_g, 0_{g\inv})$ is then well-defined and an element of $\Gamma^\s_{\tg(g)}=C_{\tg(g)}$. Denote it by $c_{\tg(g)}$. Then 
$\gamma_g=c_{\tg(g)}\cdot 0_g$ by construction, since $\tilde\iv(0_g)=0_{\iv(g)}=0_{g\inv}$.
This shows that $\Gamma^\s$ is isomorphic to the pullback vector bundle $\tg^!C\to G$, via the map \[\tg^!C\to \Gamma^\s, \qquad (c_{\tg(g)}, g)\mapsto c_{\tg(g)}\cdot 0_g,\]
and leads to the short exact sequence 
\begin{equation}\label{core_sequence} 0 \longrightarrow \tg^!C\longrightarrow \Gamma\overset{\tilde \s^!}{\longrightarrow} \s^!E\longrightarrow 0
\end{equation}
of vector bundles over $G$. The map $\tilde\s^!\colon \Gamma \to \s^!E$ is, as always, the vector bundle epimorphism $\gamma_g\mapsto (\tilde \s(\gamma_g), g)$.

Define $\pi_i\colon G\lrtimes{\s}{\tg}G\to M$ for $i=0,1,2$ by
     $\pi_2(g,h)=\s(h)$, $\pi_1(g,h)=\tg(h)=\s(g)$, and $\pi_0(g,h)=\tg(g)$. 
Consider a  splitting $\sigma\colon \s^!E\to\Gamma$ of the short exact sequence \eqref{core_sequence} of vector bundles over $G$, such that
\begin{equation*}\label{splitting_unital}
\sigma(e, 1_x)=1_{e}\in \Gamma
\end{equation*}
for all $x\in M$ and $e\in E_x$.
      In this splitting the VB-groupoid structure is equivalent to a 2-term representation up to homotopy of $G$ on $C[0]\oplus E[1]$, see \cite{GrMe17}:
      \begin{itemize}
      \item a morphism $\partial\colon C\to E$ of vector bundles (over the identity on $M$);
      \item two smooth quasi-actions\footnote{Here, $\mathfrak{gl}(C)$ is the set of linear maps between fibers of $C$. A quasi-action is a map $\Delta\colon G\to \mathfrak{gl}(C)$ such that $\Delta_g\colon C_{\s(g)}\to C_{\tg(g)}$ for all $g\in G$.}   $\Delta^C\colon G\to \mathfrak{gl}(C)$ and $\Delta^E\colon G\to \mathfrak{gl}(E)$ with $\Delta^C(1_x)=\id_{C_x}$ and $\Delta^E(1_x)=\id_{E_x}$ for all $x\in M$
      ($\Delta^C$ and $\Delta^E$ are \emph{unital});
      \item a transformation $2$-cochain
       $\Omega\in C^2(G, E\to C)=\Gamma(\operatorname{Hom}((\pi_2)^!E, (\pi_0)^!C))$ vanishing on $(1_{\tg(g)},g)$ and $(g, 1_{\s(g)})$ for all $g\in G$
      \end{itemize}
      such that 
      \begin{enumerate}
      \item $\partial\circ \Delta^C_g=\Delta^E_g\circ\partial\colon C_{\s(g)}\to E_{\tg(g)}$ for all $g\in G$;
      \item $0=\Delta^C_g\Delta^C_h-\Delta^C_{gh}+\Omega_{g,h}\circ\partial\colon C_{\s(h)}\to C_{\tg(g)}$;
       \item $0=\Delta^E_g\Delta^E_h-\Delta^E_{gh}+\partial\circ\Omega_{g,h}\colon E_{\s(h)}\to E_{\tg(g)}$;
       \item and $0=\Delta^C_g\Omega_{h,l}-\Omega_{gh,l}+\Omega_{g, hl}-\Omega_{g,h}\Delta^E_l\colon E_{\s(l)}\to C_{\tg(g)}$
      \end{enumerate}
      for $(g,h,l)\in G^{(3)}$.

      The VB-groupoid $\Gamma$ is isomorphic via the splitting to 
      \[\begin{tikzcd}
	\tg^!C\oplus \s^! E & G \\
	{E} & M
	\arrow[from=1-1, to=1-2]
	\arrow[shift left=1, from=1-1, to=2-1]
	\arrow[shift right=1, from=1-1, to=2-1]
	\arrow[shift right=1, from=1-2, to=2-2]
	\arrow[shift left=1, from=1-2, to=2-2]
	\arrow[from=2-1, to=2-2]
      \end{tikzcd}\] with structure maps
      \begin{enumerate}
      \item $\tilde \s\colon \tg^!C\oplus \s^! E \to E$, \quad $\tilde\s(c, g,e)=e$, 
      \item $\tilde \tg\colon \tg^!C\oplus \s^! E \to E$, \quad $\tilde\tg(c, g,e)=\Delta^E_ge+\partial c$,
      \item $\tilde\epsilon\colon E\to \tg^!C\oplus \s^! E $, \quad $\tilde\epsilon (e_x)=(0^C_x,1_x,e_x)$,
      \item $\tilde \m\colon (\pi_0)^!C\oplus (\pi_1)^!C\oplus (\pi_2)^!E \to \tg^!C\oplus \s^! E$,
      \[\tilde \m\left( (d, g, \Delta^E_he+\partial c), (c,h, e)\right)=(d+\Delta^C_gc-\Omega_{g,h}e, gh, e) \]
      and 
      \item $\tilde\iv\colon \tg^!C\oplus \s^! E\to \tg^!C\oplus \s^! E$, $\tilde\iota(c,g,e)=(-\Delta^C_{g\inv}c+\Omega_{g\inv, g}e, g\inv, \Delta_g^Ee+\partial c)$.
      \end{enumerate}
      Here the space $\Gamma^{(2)}\to G^{(2)}$ is identified with the vector bundle
      \[ (\pi_0)^!C\oplus (\pi_1)^!C\oplus (\pi_2)^!E\to G^{(2)}
      \]
      via 
      \[\left( (d, g, \Delta^E_he+\partial c), (c, h, e)\right) \simeq  (d, g, c, h,e)
      \]
      and the multiplication $\tilde \m$ then sends $(d,g,c,h,e)$ to $(d+\Delta^C_gc-\Omega_{g,h}e, gh, e)$.

	\subsection{Semi-direct product of a Lie groupoid with a $2$-representation}
	This section describes the Lie $2$-groupoid obtained by applying the bar construction to a split VB-groupoid.
	
\begin{proposition}\label{prop:vb-groupoids-bar-construction-layer-structure}
	Let $(\Gamma, E, G, M)$ be a VB-groupoid with core $C$. Choose a splitting $\sigma\colon \s^!E\to \Gamma$ such that \eqref{splitting_unital}.
	Then, for $r > 0$, the $r$-th layer of $W_\bullet \Gamma$ is diffeomorphic to
	\begin{equation*}
		\left\{\left.\begin{pmatrix}
				g_0 & g_1 & g_2 & \dots & g_{r-1} & \\
				& c_{1,1} & c_{1,2} & \dots & c_{1,r-1} & e_1 \\
				& & c_{2,2} & \dots & c_{2,r-1} & e_2 \\
				& & & \ddots & \vdots & \vdots \\
				& & & & c_{r-1,r-1} & e_{r-1} \\
				& & & & & e_r
			\end{pmatrix}
		\right|
		\begin{aligned}
			\s(g_{j}) &= \tg(g_{j+1}) && \forall\, 0 \leq j < r-1\\
			q_C(c_{i,j}) &= \tg(g_{j}) && \forall\, 1 \leq i \leq j < r \\
			q(e_i) &= \s(g_{r-1}) && \forall\, 1 \leq i \leq r
		\end{aligned}
		\right\}
	\end{equation*}
	as a smooth embedded submanifold  of $G^{(r)} \times E^{(r)} \times C^{(r-1)} \times C^{(r-2)} \times \dots \times C^{(1)}$.
\end{proposition}
Note that for $C=0$, this is isomorphic to the nerve $W_\bullet \Gamma$  of the action groupoid $G\times E$ since in this case $\Gamma$ is the action groupoid for the action $\Delta^E$ of $G$ on $E$.
Similarly, for $E=0$, the linear Lie groupoid $\Gamma=\tg^!C\rr M$ is the semi-direct product $G \ltimes C\rr M$.
The general case is therefore a `$2$-term' generalisation of the semi-direct product, which is referred to as \textbf{the semi-direct product of $G$ with the $2$-term representation $(C \xrightarrow{\partial} E, \Delta^C, \Delta^E, \Omega)$}.

\begin{proof}
	An element of $W_r \Gamma$ is a triangular arrangement of elements of $\Gamma$:
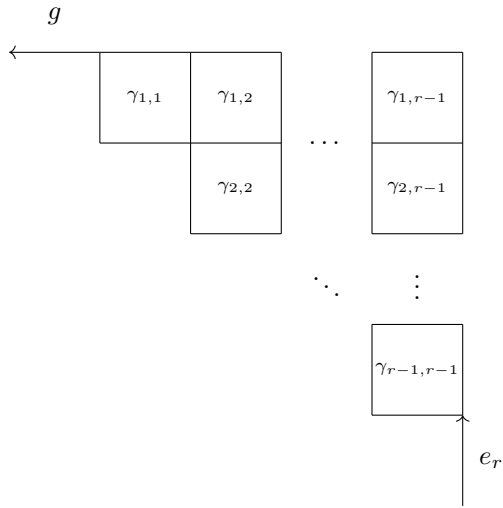
	\begin{equation*}
		\begin{tikzpicture}[y=-1.2cm,x=1.2cm]
			\node at (0.5,-0.4) {$g$};
			\draw[->] (1,0) to (0,0);

			\arrowsquare[1,0,1,1,$\gamma_{1,1}$,tlbr]

			\arrowsquare[2,0,1,1,$\gamma_{1,2}$,tbr]
			\arrowsquare[2,1,1,1,$\gamma_{2,2}$,lbr]

			\node at (3.5,1) {$\dots$};
			\arrowsquare[4,0,1,1,$\gamma_{1,r-1}$,tlr]
			\arrowsquare[4,1,1,1,$\gamma_{2,r-1}$,tlbr]

			\node at (3.5,2.5) {$\ddots$};
			\node at (4.5,2.5) {$\vdots$};
			\arrowsquare[4,3,1,1,$\gamma_{r-1,r-1}$,tlbr]
	
			\node[text width=1cm] at (5.6, 4.5) {$e_r$};
			\draw[->] (5,5) to (5,4);
		\end{tikzpicture}
	\end{equation*}
	Since all horizontal arrows in the same column are equal and elements of $\Gamma$ are triples in $\tg^! C \oplus \s^! E$, rewrite this as
	
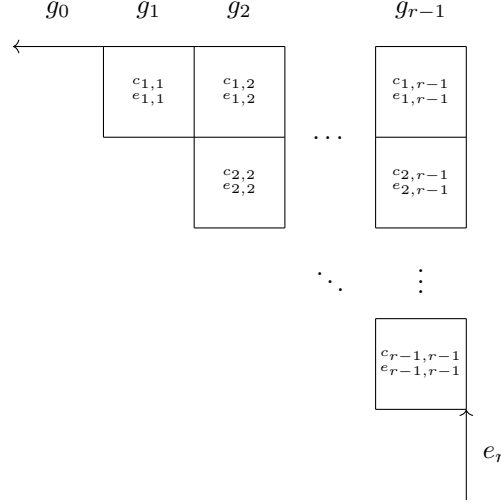
\begin{equation*}
		\begin{tikzpicture}[y=-1.2cm,x=1.2cm]
			\node at (0.5,-0.4) {$g_0$};
			\node at (1.5,-0.4) {$g_1$};
			\node at (2.5,-0.4) {$g_2$};
			\node at (4.5,-0.4) {$g_{r-1}$};

			\draw[->] (1,0) to (0,0);

			\arrowsquare[1,0,1,1,$\substack{c_{1,1}\\e_{1,1}}$,tlbr]

			\arrowsquare[2,0,1,1,$\substack{c_{1,2}\\e_{1,2}}$,tbr]
			\arrowsquare[2,1,1,1,$\substack{c_{2,2}\\e_{2,2}}$,lbr]

			\node at (3.5,1) {$\dots$};
			\arrowsquare[4,0,1,1,$\substack{c_{1,r-1}\\e_{1,r-1}}$,tlr]
			\arrowsquare[4,1,1,1,$\substack{c_{2,r-1}\\e_{2,r-1}}$,tlbr]

			\node at (3.5,2.5) {$\ddots$};
			\node at (4.5,2.5) {$\vdots$};
			\arrowsquare[4,3,1,1,$\substack{c_{r-1,r-1}\\e_{r-1,r-1}}$,tlbr]
	
			\node[text width=1cm] at (5.6, 4.5) {$e_r$};
			\draw[->] (5,5) to (5,4);
		\end{tikzpicture}
	\end{equation*}
	This information is recorded in the matrix
	\begin{equation*}
		\begin{pmatrix}
			g_0 & g_1 & g_2 & \dots & g_{r-1} & \\
			& c_{1,1} & c_{1,2} & \dots & c_{1,r-1} & e_1 \\
			& & c_{2,2} & \dots & c_{2,r-1} & e_2 \\
			& & & \ddots & \vdots & \vdots \\
			& & & & c_{r-1,r-1} & e_{r-1} \\
			& & & & & e_r
		\end{pmatrix},
	\end{equation*}
	where $e_i:=e_{i,r-1}$ for $i=1,\ldots, r-1$.
	All elements $e_{i,j}\in E$ with $0 < i\leq  j < r-1$ can be calculated recursively as
	\begin{equation*}
		e_{i,j} = \tilde{\s}(\gamma_{i,j}) = \tilde{\tg }(\gamma_{i, j+1}) = \partial c_{i,j+1} + \Delta^E_{g_{j+1}} e_{i,j+1}\text{,}
	\end{equation*}
	so this process loses no information and is invertible.
\end{proof}

By \cite{MeTa11} the $0$-th layer $W_0\Gamma$ is
\begin{equation*}
	W_0\Gamma
	= M
\end{equation*}
and the first layer is 
\begin{equation*}
	W_1\Gamma
	= G \lrtimes{\s}{q} E
	= \s^! E
\end{equation*}
as a vector bundle over $G$.
Its elements are edges
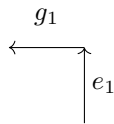
\begin{equation*}
	\begin{tikzpicture}[y=-1cm]
		\node at (0.5, -0.4){$g_1$};
		\draw[->] (1,0) to (0,0); 
		\draw[->] (1,1) to (1,0); 
		\node[text width=0.8cm] at (1.5, 0.5) {$e_1$};
	\end{tikzpicture}
\end{equation*}
or matrices
\begin{equation*}
	\begin{pmatrix}
		g_1 &  \\
		& e_1
	\end{pmatrix}
\end{equation*}
with $\s(g_1) = q(e_1)$.
The second layer is given by 
\begin{align*}
	W_2 \Gamma
	&= G^{(2)} \lrtimes{p_2}{q_\Gamma} \,\Gamma \lrtimes{\tilde{\s}}{p_1} E^{(2)}
	\cong G^{(2)} \lrtimes{p_2}{q_\Gamma} \,(\tg^! C \oplus \s^! E) \lrtimes{\tilde \s}{p_1} E^{(2)},
\end{align*}
where $p_i$ is the projection of $G^{(r)}$ or $E^{(r)}$ to its $i$-th component from the left.
By \cite{MeTa11} the simplex notation of its elements is
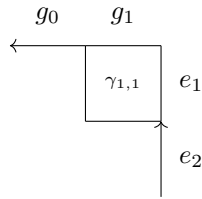
\begin{equation*}
	\begin{tikzpicture}[y=-1cm]
		\node at (0.5,-0.4) {$g_0$};
		\node at (1.5,-0.4) {$g_1$};
		\draw[->] (1,0) to (0,0);
		\arrowsquare[1,0,1,1,$\gamma_{1,1}$,tlbr]
		\node at (2.4,0.5) {$e_1$};
		\node at (2.4,1.5) {$e_2$};
		\draw[->] (2,2) to (2,1); 
	\end{tikzpicture}
\end{equation*}
With $\gamma_{1,1} = (c_{1,1}, g_1, e_{1})$, this is written as a matrix 
\begin{equation*}
	\begin{pmatrix}
		g_0 & g_1 & \\
		& c_{1,1} & e_1 \\
		& & e_2
	\end{pmatrix}
	\qquad
	\text{ with }
	\qquad
	\begin{aligned}
		q(e_1) &= \s(g_1)\text{,}
		& q_C(c_{1,1}) &= \s(g_0) = \tg(g_1)\text{,}\\
		q(e_2) &= \s(g_1)\text{,}
		& \s(g_0) &= \tg(g_1)\text{.}
	\end{aligned}
\end{equation*}
Lastly, the third layer consists of simplices
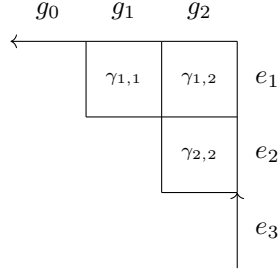
\begin{equation*}
	\begin{tikzpicture}[y=-1cm]
		\node at (0.5,-0.4) {$g_0$};
		\node at (1.5,-0.4) {$g_1$};
		\node at (2.5,-0.4) {$g_2$};

		\draw[->] (1,0) to (0,0);
		\arrowsquare[1,0,1,1,$\gamma_{1,1}$,tlbr]
		\arrowsquare[2,0,1,1,$\gamma_{1,2}$,tbr]
		\arrowsquare[2,1,1,1,$\gamma_{2,2}$,lbr]
		\draw[->] (3,3) to (3,2);

		\node at (3.4,0.5) {$e_{1}$};
		\node at (3.4,1.5) {$e_{2}$};
		\node at (3.4,2.5) {$e_{3}$};
	\end{tikzpicture}
\end{equation*}
or matrices
\begin{equation*}
	\begin{pmatrix}
		g_0 & g_1 & g_2 & \\
		& c_{1,1} & c_{1,2} & e_1 \\
		& & c_{2,2} & e_2 \\
		& & & e_3
	\end{pmatrix}
	\qquad
	\text{ with }
	\qquad
	\begin{aligned}
	q(e_1) &= \s(g_2)\text{,}
	& q_C(c_{1,1}) &= \tg(g_1)\text{,}
	& \s(g_0) &= \tg(g_1)\text{,}
	\\
	q(e_2) &= \s(g_2)\text{,}
	& q_C(c_{1,2}) &= \tg(g_2)\text{,}
	& \s(g_1) &= \tg(g_2)\text{,}
	\\
	q(e_3) &= \s(g_2)\text{,}
	& q_C(c_{2,2}) &= \tg(g_2)\text{.}
	&
	\end{aligned}
\end{equation*}

\bigskip

A simplex
\begin{equation*}\label{eq:vb-groupoid-M-r}
	x := \begin{pmatrix}
		g_0 & g_1 & g_2 & \dots & g_{r-1} & \\
		& c_{1,1} & c_{1,2} & \dots & c_{1,r-1} & e_1 \\
		& & c_{2,2} & \dots & c_{2,r-1} & e_2 \\
		& & & & c_{r-1,r-1} & e_{r-1} \\
		& & & & & e_r
	\end{pmatrix}	\in W_{r}\Gamma
\end{equation*}
for $0 < i < r$ is sent by $f^r_i$ to
\begin{equation*}
	\begin{pmatrix}
		g_0 & g_1 & g_2 & \dots & g_{i-2} & g_{i-1}\cdot g_{i} & g_{i+1} & g_{i+2} & \dots & g_{r-1} & \\
		& c_{1,1} & c_{1,2} & \dots & c_{1,i-2} &d^C_1& c_{1,i+1} & c_{1,i+2} & \dots & c_{1,r-1} & e_1 \\
		& & c_{2, 2} & \dots & c_{2,i-2} & d^C_2 & c_{2,i+1} & c_{2,i+2} & \dots & c_{2,r-1} & e_2 \\
		& & & \ddots & \vdots & \vdots &  \vdots & \vdots & \ddots & \vdots & \vdots \\
		& & & & c_{i-2, i-2}& d^C_{i-2} & c_{i-2,i+1} & c_{i-2,i+2} & \dots & c_{i-2,r-1}& e_{i-2} \\
		& & & & & d^C_{i-1} & c_{i-1,i+1} & c_{i-1,i+2} & \dots & c_{i-1,r-1} & e_{i-1} \\
		& & & & & & d^E_{i+1} & d^E_{i+2} & \dots &d^E_{r-1}& e_{i} + e_{i+1} \\
		& & & & & & & c_{i+1,i+1} & \dots & c_{i+1,r-1} & e_{i+1} \\
		& & & & & & & & \ddots & \vdots & \vdots \\
		& & & & & & & & & c_{r-1,r-1} & e_{r-1} \\
		& & & & & & & & & & e_r 
	\end{pmatrix}\in W_{r-1}\Gamma
\end{equation*}
 where
\begin{equation*}
	d_{j}^C := c_{j,i-1} + \Delta^C_{g_{i-1}} c_{j,i} - \Omega_{g_{i-1},g_i} \left( \left( \sum_{k=i}^{r-2} \Delta^E_{g_{i+1}}\cdots \Delta^E_{g_{k}}\partial c_{j,k+1}\right) + \Delta^E_{g_{i+1}} \dots \Delta^E_{g_{r-1}} e_j\right)
\end{equation*}
for $1 \leq j < i$ and
\begin{equation*}
	d_{j}^E := c_{i,j}+c_{i+1,j}
\end{equation*}
for $i < j < r$.
The map $f^r_r\colon W_r \Gamma \to W_{r-1}\Gamma$ sends the simplex $x$ above to the matrix
\begin{equation*}
	\begin{pmatrix}
		g_0 & g_1 & g_2 & \dots & g_{r-2} & \\
		& c_{1,1} & c_{1,2} & \dots & c_{1,r-2} & \partial c_{1,r-1} + \Delta^E_{g_{r-1}} e_1 \\
		& & c_{2,2} & \dots & c_{2,r-2} & \partial c_{2,r-1} + \Delta^E_{g_{r-1}} e_2 \\
		& & & \ddots & \vdots & \vdots \\
		& & & & c_{r-2,r-2} & \partial c_{r-2,r-1} + \Delta^E_{g_{r-1}} e_{r-2} \\
		& & & & & \partial c_{r-1, r-1} + \Delta^E_{g_{r-1}} e_{r-1}
	\end{pmatrix}
\end{equation*}
and $f^r_0\colon W_r\Gamma \to W_{r-1}\Gamma$ sends $x$ to
\begin{equation*}
	\begin{pmatrix}
		g_1 & g_2 & g_3 & \dots & g_{r-1} & \\
		& c_{2,2} & c_{2,3} & \dots & c_{2,r-1} & e_2 \\
		& & c_{3,3} & \dots & c_{3,r-1} & e_3 \\
		& & & \ddots & \vdots & \vdots \\
		& & & & c_{r-1,r-1} & e_{r-1} \\
		& & & & & e_r
	\end{pmatrix}	
\end{equation*}
and simply deletes the first $G$-element and $C$-row.
In the lowest layers, the two maps 
\begin{align*}
	f^1_0\colon &W_1\Gamma \to W_0\Gamma, \quad 
	\begin{pmatrix}
		g_0 & \\ & e_1
	\end{pmatrix}
	\mapsto \s(g_0) = q(e_1), \\
	f^1_1\colon &W_1\Gamma \to W_0\Gamma, \quad 
	\begin{pmatrix}
		g_0 & \\ & e_1
	\end{pmatrix}
	\mapsto \tg(g_0)
\end{align*}
are defined as usual.
Then
\begin{align*}
	f^2_0\colon &W_2\Gamma \to W_1 \Gamma,\quad 
	\begin{pmatrix}
		g_0 & g_1 & \\
		& c_{1,1} & e_1 \\
		& & e_2
	\end{pmatrix} \mapsto
	\begin{pmatrix}
		g_1 & \\ & e_2
	\end{pmatrix},
	\\
	f^2_1\colon &W_2 \Gamma \to W_1\Gamma,\quad 
	\begin{pmatrix}
		g_0 & g_1 & \\
		& c_{1,1} & e_1 \\
		& & e_2
	\end{pmatrix} \mapsto
	\begin{pmatrix}
		g_0 \cdot g_1 & \\ & e_1 + e_2
	\end{pmatrix},
	\\
	f^2_2\colon &W_2 \Gamma \to W_1\Gamma,\quad 
	\begin{pmatrix}
		g_0 & g_1 & \\
		& c_{1,1} & e_1 \\
		& & e_2
	\end{pmatrix} \mapsto
	\begin{pmatrix}
		g_0 & \\ & \partial c_{1,1} + \Delta^E_{g_1} e_1
	\end{pmatrix}
\end{align*}
and the third layer face maps send the simplex 
\begin{equation*}
	y:=
	\begin{pmatrix}
		g_0 & g_1 & g_2 & \\
		& c_{1,1} & c_{1,2} & e_1 \\
		& & c_{2,2} & e_2 \\
		& & & e_3
	\end{pmatrix}
\end{equation*} to the following images:
\begin{align*}
	f^3_0\colon W_3 \Gamma \to W_2 \Gamma, \quad 
y
	&\mapsto
	\begin{pmatrix}
		g_1 & g_2 & \\
		& c_{2,2} & e_2 \\
		& & e_3
	\end{pmatrix}\text{ ,}
	\\
	f^3_1\colon W_3 \Gamma \to W_2\Gamma,  \quad 
	y
	&\mapsto
	\begin{pmatrix}
		g_0 \cdot g_1 & g_2 & \\
		& c_{1,2} + c_{2,2} & e_1 + e_2 \\
		& & e_3
	\end{pmatrix},
	\\
	f^3_2\colon W_3 \Gamma \to W_2 \Gamma,  \quad 
	y
	&\mapsto
	\begin{pmatrix}
		g_0 & g_1 \cdot g_2 & \\
		& c_{1,1} + \Delta^C_{g_1} c_{1,2} - \Omega_{g_1, g_2} e_1 & e_1 \\
		& & e_2 + e_3
	\end{pmatrix}\text{ ,}
	\\
	f^3_3\colon W_3 \Gamma \to W_2 \Gamma, \quad 
	y
	&\mapsto
	\begin{pmatrix}
		g_0 & g_1 & \\
		& c_{1,1} & \partial c_{1,2} + \Delta^E_{g_2} e_1 \\
		& & \partial c_{2,2} + \Delta^E_{g_2} e_2
	\end{pmatrix}\text{.}
\end{align*}

The surmersive $(2,i)$-horn maps are hence given 
for $i=0,1,2$ by 
\begin{align*}
	\lambda_{2,0}&=(f^2_1, f^2_2)\colon \begin{pmatrix}
		g_0 & g_1 & \\
		& c_{1,1} & e_1 \\
		& & e_2
	\end{pmatrix} \mapsto
	\left(\begin{pmatrix}
		g_0 \cdot g_1 & \\ & e_1 + e_2
	\end{pmatrix},
		\begin{pmatrix}
		g_0 & \\ & \partial c_{1,1} + \Delta^E_{g_1} e_1
	\end{pmatrix}
	\right),\\
	\lambda_{2,1}&=(f^2_0, f^2_2)\colon \begin{pmatrix}
		g_0 & g_1 & \\
		& c_{1,1} & e_1 \\
		& & e_2
	\end{pmatrix}  \mapsto
	\left(
	\begin{pmatrix}
		g_1 & \\ & e_2
	\end{pmatrix},
		\begin{pmatrix}
		g_0 & \\ & \partial c_{1,1} + \Delta^E_{g_1} e_1
	\end{pmatrix}
	\right), \\
\text{ and }\quad  	\lambda_{2,2}&=(f^2_0, f^2_1)\colon \begin{pmatrix}
		g_0 & g_1 & \\
		& c_{1,1} & e_1 \\
		& & e_2
	\end{pmatrix} \mapsto
	\left(
	\begin{pmatrix}
		g_1 & \\ & e_2
	\end{pmatrix},
	\begin{pmatrix}
		g_0 \cdot g_1 & \\ & e_1 + e_2
	\end{pmatrix}
	\right).
\end{align*}

\bigskip

For $0 < i < r$, the general $i$-th degeneracy map of the $r$-th layer, $\delta_i^r\colon W_r \Gamma \to W_{r+1} \Gamma$ maps the basic simplex $x$ to 
\begin{equation*}
	\begin{pmatrix}
		g_0 & g_1 & g_2 & \dots & g_{i-1} & 1^G_{t(g_i)}& g_{i} & g_{i+1} & \dots & g_{r-1} &  \\
		& c_{1,1} & c_{1,2} & \dots & c_{1,i-1} & 0^C_{t(g_i)}& c_{1,i} & c_{1,i+1} & \dots & c_{1,r-1} & e_1 \\
		&  & c_{2,2} & \dots & c_{2,i-1} & 0^C_{t(g_i)} & c_{2,i} & c_{2,i+1} & \dots & c_{2,r-1} & e_2 \\
		& & & \ddots & \vdots & \vdots & \vdots & \vdots & \ddots & \vdots & \vdots \\
		& & & & c_{i-1,i-1} & 0^C_{t(g_i)}& c_{i-1,i} & c_{i-1,i+1} & \dots & c_{i-1,r-1} & e_{i-1}\\
		& & & & & 0^C_{t(g_i)}& c_{i,i} & c_{i,i+1} & \dots & c_{i,r-1}& e_{i}\\
		& & & & & & 0^C_{t(g_{i})} & 0^C_{t(g_{i+1})}& \dots & 0^C_{t(g_{r-1})}& 0^E_{s(g_{r-1})}\\
		& & & & & & & c_{i+1,i+1} & \dots & c_{i+1,r-1} & e_{i+1}\\
		& & & & & & & & \ddots & \vdots & \vdots \\
		& & & & & & & & & c_{r-1,r-1} & e_{r-1} \\
		& & & & & & & & & & e_{r}
	\end{pmatrix}
\end{equation*}
in $W_{r+1}\Gamma$.

Applying $\delta_0^r\colon W_{r} \Gamma \to W_{r+1} \Gamma$ to the simplex $x$ yields \begin{equation*}
	\delta_0^r(x) = \begin{pmatrix}
		1^G_{\tg(g_0)} & g_0 & g_1 & g_2 & \dots & g_{r-1} &  \\
		& 0^C_{\tg(g_0)} & 0^C_{\tg(g_1)} & 0^C_{\tg(g_2)}& \dots &0^C_{\tg(g_{r-1})}& 0^E_{\s(g_{r-1})}\\
		& & c_{1,1} & c_{1,2} & \dots & c_{1,r-1} & e_1 \\
		& & & c_{2,2} & \dots & c_{2,r-1} & e_2 \\
		& & & & \ddots & \vdots & \vdots \\
		& & & & & c_{r-1,r-1} & e_{r-1}\\
		& & & & & & e_{r}
	\end{pmatrix},
\end{equation*}
and the degeneracy map $\delta_r^r\colon W_{r} \Gamma \to W_{r+1} \Gamma$ sends $x$ to
\begin{equation*}
	\begin{pmatrix}
		g_0 & g_1 & g_2 & \dots & g_{r-1} & 1^G_{\s(g_{r-1})}  & \\
		& c_{1,1} & c_{1,2} & \dots & c_{1,r-1} & 0^C_{\s(g_{r-1})} & e_1\\
		& & c_{2,2} & \dots & c_{2,r-1} & 0^C_{\s(g_{r-1})}& e_2 \\
		& & & \ddots & \vdots & \vdots& \vdots \\
		& & & & c_{ r-1,r-1} & 0^C_{\s(g_{r-1})} & e_{r-1}\\
		& & & & & 0^C_{\s(g_{r-1})}& e_r \\
		& & & & & & 0^E_{\s(g_{r-1})}
	\end{pmatrix}.\end{equation*}

The degeneracy maps on lower layers 
read \begin{equation*}
	\delta^0_0\colon M \to W_1 \Gamma, \quad  x \mapsto \begin{pmatrix}
		1^G_x & \\ & 0^E_x
	\end{pmatrix},
\end{equation*}
\begin{align*}
	\delta^1_0\colon &W_1 \Gamma \to W_2 \Gamma, \quad 
	\begin{pmatrix}
		g_0 & \\ & e_1
	\end{pmatrix}
	\mapsto
	\begin{pmatrix}
		1^G_{\tg(g_0)} & g_0 & \\
		& 0^C_{\tg(g_0)} & 0^E_{\s(g_0)} \\
		& & e_1 
	\end{pmatrix},
	\\
	\delta^1_1\colon &W_1 \Gamma \to W_2  \Gamma,\quad 
	\begin{pmatrix}
		g_0 & \\ & e_1
	\end{pmatrix}
	\mapsto
	\begin{pmatrix}
		g_0 & 1^G_{\s(g_0)} & \\ 
		& 0^C_{\s(g_0)} & e_1 \\
		& & 0^E_{\s(g_0)}
	\end{pmatrix},
\end{align*}
etc.

  \def\cprime{$'$} \def\polhk#1{\setbox0=\hbox{#1}{\ooalign{\hidewidth
  \lower1.5ex\hbox{`}\hidewidth\crcr\unhbox0}}} \def\cprime{$'$}
  \def\cprime{$'$}

	\end{document}